%% file: topDerSurface_GanglSturm.tex
\documentclass[12pt, twoside]{article}

\usepackage{amsmath,amsthm}

\usepackage[utf8]{inputenc}

\usepackage{enumerate}
\usepackage{esint}

\usepackage{fancyhdr}
\usepackage{cite}
\usepackage{enumitem}

\usepackage{graphics}

\usepackage{xargs}                      
\usepackage[colorinlistoftodos,prependcaption,textsize=tiny]{todonotes}
\usepackage[normalem]{ulem}

\newcommandx{\pcomment}[2][1=]{\todo[linecolor=red,backgroundcolor=red!25,bordercolor=red,#1]{#2}}
\newcommandx{\kcomment}[2][1=]{\todo[linecolor=blue,backgroundcolor=blue!25,bordercolor=blue,#1]{#2}}

\newif\ifpreprint

\usepackage{hyperref}
\hypersetup{%
    colorlinks=True, 
    citecolor=blue,
    linkcolor=black}

\usepackage[charter]{mathdesign}

\usepackage{accents}

\newcommand{\rom}[1]{\uppercase\expandafter{\romannumeral #1\relax}}

\newcommand{\black}{\color{black}}

  \theoremstyle{definition}
  \newtheorem{theorem}{Theorem}[section]
  \newtheorem{corollary}[theorem]{Corollary}
  
  \newtheorem{lemma}[theorem]{Lemma}
  \newtheorem{definition}[theorem]{Definition}
  \newtheorem{remark}[theorem]{Remark}

  \newtheorem*{assumption*}{Assumption}

  \numberwithin{equation}{section}
    
  \newcommand{\subjclass}[1]{\bigskip\noindent\emph{2010 Mathematics Subject Classification:}\enspace#1}
  \newcommand{\keywords}[1]{\noindent\emph{Keywords:}\enspace#1}

\input{slidesymbols}

\setenumerate[0]{label=(\alph*)}

\newcommand{\eps}{{\varepsilon}}

\newcommand{\overbar}[1]{\mkern 1.5mu\overline{\mkern-1.5mu#1\mkern-1.5mu}\mkern 1.5mu}

\newcommand{\ben}{\begin{equation}}
\newcommand{\een}{\end{equation}}

\newcommand{\benn}{\begin{equation*}}
\newcommand{\eenn}{\end{equation*}}

\usepackage[colorinlistoftodos]{todonotes}

\newcommand{\tow}{\rightharpoonup}

\usepackage{authblk}
\usepackage[T1]{fontenc}
\usepackage[utf8]{inputenc}

\author[1]{Peter Gangl\thanks{E-Mail: gangl(at)math.tugraz.at}}
\author[2]{Kevin Sturm\thanks{E-Mail: kevin.sturm(at)tuwien.ac.at}}

\affil[1]{TU Graz, Steyrergasse 30/III, 
    8010 Graz, Austria}
\affil[2]{TU Wien, Wiedner Hauptstr. 8-10,
      1040 Vienna, Austria}

\title{Topological derivative for PDEs on surfaces}

\date{\today}

\begin{document}

\maketitle

\begin{abstract}
    In this paper we study the problem of the optimal distribution of two materials on smooth submanifolds $M$ of dimension $d-1$ in $\VR^d$ without boundary by means of the topological derivative. We consider a class of shape optimisation problems which are constrained by a linear partial differential equation on the surface. We examine the singular perturbation of the differential operator and material coefficients and derive the topological derivative. Finally, we show how the topological derivative in conjunction with a level set method on the surface can be used to solve the topology optimisation problem numerically.


\end{abstract}

  \maketitle

\section{Introduction}

The topological derivative of a shape function $\mathcal J = \mathcal J(\Omega)$ at a point $q \in \Omega$ measures the sensitivity of $\mathcal J$ with respect to a singular perturbation of the domain $\Omega$. The concept was first introduced in the context of mechanical engineering in \cite{EschenauerKobelevSchumacher1994}, and later introduced in a mathematically rigorous way in \cite{SokolowskiZochowski1999,GarreauGuillaumeMasmoudi2001}. We refer the reader to \cite{b_NOSO_2013a} for a thorough introduction to the concept of topological derivatives and many applications. 

In all of these previous works the partial differential equation is always defined on an open subset of $\VR^d$. However, in a number of applications the arising partial differential equation is defined on submanifolds of $\VR^d$. Related works are \cite{deng2020topology}, where a topology optimisation problem of surface flows is studied by means of a material interpolation approach, and \cite{KANG201618} where topological derivatives for shell structures were derived heuristically and used in an iterative algorithm. In contrast to PDEs defined on open subsets of $\VR^d$, geometrical 
properties of the manifold emerge when performing a singular perturbation of a surface PDE. Let us mention \cite{MR2792999} where the topological perturbation on the boundary of a PDE defined on a domain is performed, which is 
related to our work. We also refer to \cite{a_HILA_2008a,a_HILANO_2011a} where the topological sensitivity of an electrical impedance model is performed, however the equations are defined in a 
subdomain of $\VR^d$ rather than a manifold. 

Let $(M,g)$ be a compactly embedded submanifold in $\VR^d$ of dimension $d-1$ equipped  
with the Euclidean metric $g$ of $\VR^d$. The associated Riemannian distance is denoted by $\mathfrak d:M\times M\to \VR$. The submanifold is as usual equipped with the subspace topology. The main subject of this paper is the derivation of the topological derivative of a shape optimisation problem which is constrained by a partial differential equation (PDE) on the surface $M$. We are interested in problems of the form
\begin{align}
    \underset{(\Omega,u)}{\mbox{min }} &\, J(\Omega, u) \label{eq:J_intro} \\
    \mbox{s.t. } u \in H^1(M): 
\int_M \beta_{\Omega}& \nabla^M u \cdot \nabla^M v + \gamma_\Omega uv \;dx = \int_M f_\Omega v\;dx \qquad \mbox{for all } v \in H^1(M),  \label{eq:state_intro}
\end{align}
where $H^1(M)$ denotes the space of square integrable functions with square integrable weak derivative on $M$. Here, $\Omega \subset M$ is an admissible open subset of $M$ and the functions 
\ben \label{eq_defBetaOmega}
    \beta_\Omega  = \beta_1\chi_\Omega + \beta_2\chi_{M\setminus \Omega},
    \qquad \gamma_\Omega  = \gamma_1\chi_\Omega + \gamma_2\chi_{M\setminus \Omega},
    \qquad f_\Omega = f_1 \chi_\Omega + f_2 \chi_{M\setminus \Omega},
\een
 are piecewise constant on $M$ with $\beta_1,\beta_2,\gamma_1,\gamma_2> 0$ and $f_1, f_2 \in \VR$. 
 The symbol $\nabla^M $ denotes the surface gradient on $M$. Let $\omega\subset \VR^{d-1}$ be a connected domain 
containing the origin. For a point $q \in M$ we denote the corresponding tangent space by $T_q M$. Given a point $q\in M$ we introduce $T_\eps(x):= \exp_q(E(\eps x))$, where $\exp_q$ denotes the exponential map at $q\in M$ and $E:\VR^{d-1}\to T_qM$ is an isomorphism. For a given shape $\Omega \subset M$, we denote the unique solution to \eqref{eq:state_intro} by $u(\Omega)$ and the reduced cost function by $\Cj(\Omega):= J(\Omega,u(\Omega))$. Denoting $\omega_\eps := T_\eps(\omega)$, the goal of this paper is the rigorous derivation of the topological derivative
\ben \label{eq:TD_intro}
    d \Cj(\Omega)(q) := \begin{cases}
                      \underset{\eps \searrow 0}{\mbox{lim}} \, \frac{\Cj(\Omega \cup \omega_\eps) - \Cj(\Omega)}{|\omega_\eps|} & q \in M \setminus \Omega, \\
                      \underset{\eps \searrow 0}{\mbox{lim}} \, \frac{\Cj(\Omega\setminus \overbar \omega_\eps) - \Cj(\Omega)}{|\omega_\eps|} & q \in \Omega.
                     \end{cases}
\een

\paragraph*{Structure of the paper}
In Section~\ref{sec:M}, we introduce the setting needed to deal with singular perturbations of manifolds and show some auxiliary results. These results will be used for the rigorous derivation of the topological derivative \eqref{eq:TD_intro} for a class of cost functions. Finally, in Section~\ref{sec:numerics}, we present numerical results using the topological derivative obtained in Section~\ref{sec:topo_deriv} and show its pertinence.

\subsection*{Notation and definitions}

For an open set $\Dsf\subset \VR^d$ we denote by $L_2(\Dsf)$ and $H^1(\Dsf)$ the standard $L_2$ space and Sobolev space. We equip $\VR^d$ with the Euclidean norm $| \cdot |$ and use the same notation for the corresponding matrix (operator) norm.

Let $M\subset \VR^{d-1}$ be an embedded submanifold without boundary. We denote by $\nabla^M f$ the tangential gradient of a function $f\in H^1(M)$.

\section{Preliminaries for PDEs posed on surfaces}\label{sec:M}
We collect some results which will be helpful for analysing the sensitivity of a shape functional $\Cj =\Cj(\Omega)$ with respect to a singular perturbation of the subset $\Omega \subset M$ of the manifold $M$. The key ingredient will be the fact that the exponential map associated to a point $q\in M$ in the manifold is locally diffeomorphic between the tangent space $T_qM$ and the manifold $M$.

\subsection{Singular perturbation}

For a given point $q\in M$ we denote by $\exp_q:T_qM \to M $ the exponential map associated with the manifold $M$. From now on we choose an orthonormal basis $\{v_1,\ldots, v_{d-1}\}$ of $T_qM$ with respect to the metric $g(\cdot,\cdot)$, which naturally induces an isomorphism $E:\VR^{d-1} \to T_qM$. This isomorphism $E$ 
is given by $(\alpha_1,\ldots, \alpha_{d-1}) \mapsto 
\sum_{i=1}^{d-1}\alpha_i v_i$ and hence can be written with $V:= (v_1,\ldots, v_{d-1})$ as 
$\alpha \mapsto V\alpha$ with $\alpha =(\alpha_1,\ldots, \alpha_{d-1})$. Notice that the 
norm is preserved, i.e. $\|E\alpha\|_{T_qM} = |\alpha|$ for all $\alpha\in \VR^{d-1}$ where $|\cdot|$ denotes the Euclidean norm on $\VR^{d-1}$. In this sense we will identify the tangent space $T_qM$ at a point $q\in M$ with $\VR^{d-1}$. 
In view of $d\exp_q(0)=\text{id}_{T_qM}$ on $T_qM$ the inverse function theorem implies 
that $\exp_q$ is a diffeomorphism from a ball $B_\delta(0)\subset T_qM$ of radius $\delta$ onto an open subset $U_q$ of $q$ in $M$. We will denote the pre-image under $E$ of this ball $B_\delta(0)$ by $B := E^{-1}(B_\delta(0))$ such that $B_\delta(0) = E(B)$. Finally, we define the mapping
\ben \label{eq_defTeps}
\begin{split}
    T_\eps : \VR^{d-1} &\rightarrow M \\
    x &\mapsto \exp_q(E(\eps x)).
\end{split}
\een
Figure \ref{fig_trafos} shows an illustration of all involved mappings and sets.

\begin{figure}
    \begin{center}
        \includegraphics[width=.7\textwidth]{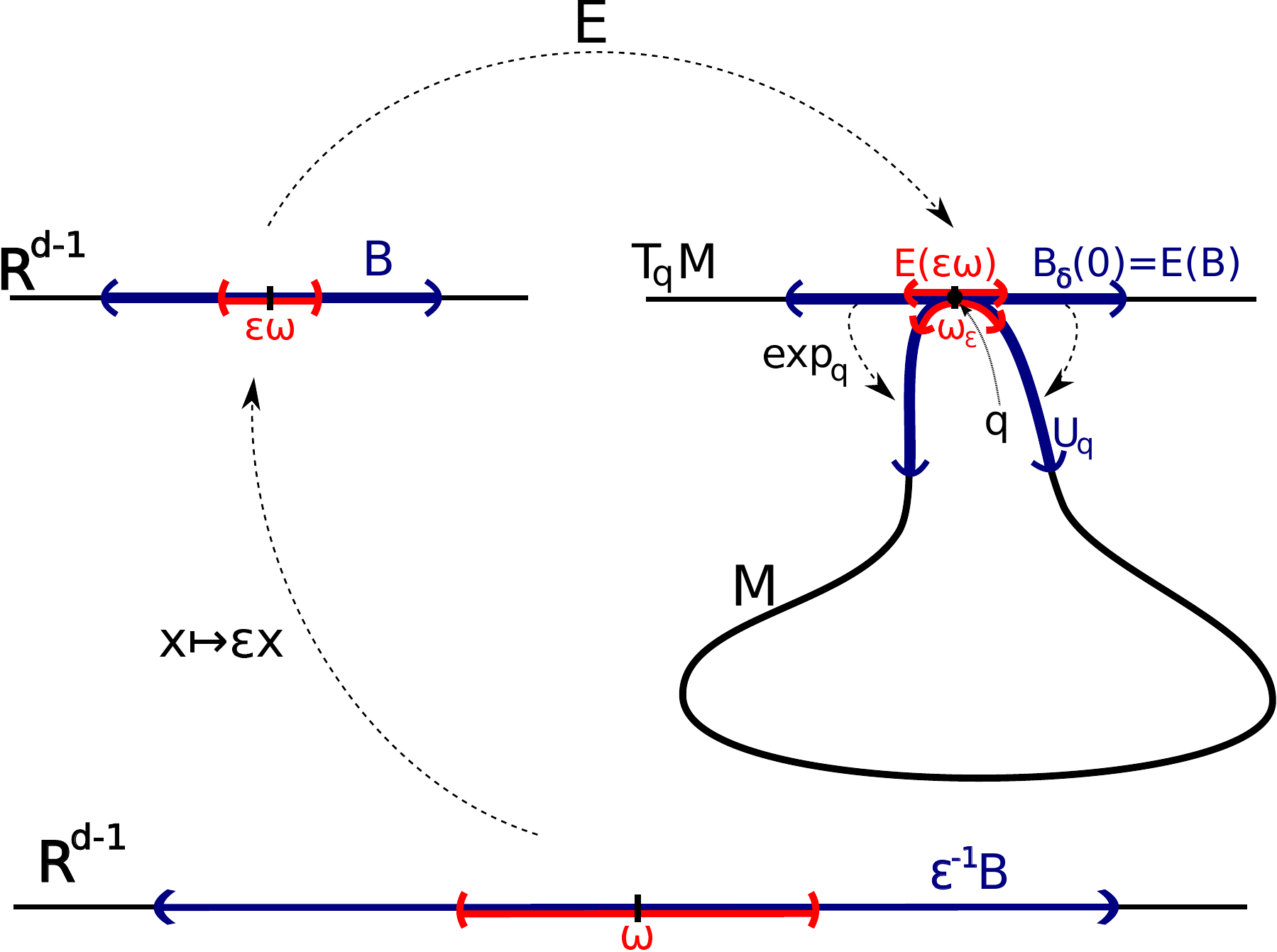}
    \end{center}
    \caption{Transformations used throughout this paper illustrated for the case $d=2$ with a manifold $M \subset \VR^2$. For a point $q \in M$, the tangent space $T_q M$ is identified with $\VR^{d-1}$ via an isomorphism $E$. The exponential map $\exp_q$ associated to the point $q\in M$ is a diffeomorphism between a ball $E(B) \subset T_qM$ and a neighborhood $U_q \subset M$ of $q$ in $M$. The mapping $T_\eps(x) = (\exp_q \circ E)(\eps x)$ maps from $\VR^{d-1}$ to the manifold $M$ and it holds that $T_\eps(\omega) = \omega_\eps$.}
    \label{fig_trafos}
\end{figure}

\begin{definition}
    Let $\omega\subset \VR^{d-1}$ be an open, bounded and connected set with $0\in \omega$. At a point $q\in M$, we define the geodesic perturbation of $M$ with respect to $\omega$ by
    \ben\label{eq:omega_eps}
\omega_\eps := T_\eps(\omega) =  \exp_q(E(\eps\omega)), \quad \eps >0.
\een
We note that in the case of the unit ball $\omega = B_1(0)\subset \VR^{d-1}$, we obtain the geodesic ball
\ben
\omega_\eps = \{x\in M:\; \mathfrak d(x,q)<\eps\}.
\een
\end{definition}

\begin{remark}
    Instead of defining the perturbation $\omega_\eps$ using the exponential map as in \eqref{eq:omega_eps} we could also work with so called retractions; see \cite[p.55, Def. 4.1.1]{b_ABMASE_2008a}. 
A retraction at $q\in M$ is a mapping $R_q:B_\delta(0)\subset T_qM \to M$ for $\delta >0$ small satisfying the following
two properties:
\begin{itemize}
    \item[(i)] $R_q(0) = q$
    \item[(ii)] $dR_q(0) = \text{id}_{T_qM}$ with the identification $T_0(T_qM) = T_qM$. 
\end{itemize}
Condition (ii) guarantees that $R_q$ is a local diffeomorphism. It can be readily checked that all 
subsequent computations remain valid replacing the exponential map $\exp_q$ by a retraction 
$R_q$ and accordingly replacing $\omega_\eps$ defined in \eqref{eq:omega_eps} by $\omega_\eps := R_q(E(\eps\omega))$ for $\eps >0$.
\end{remark}

We now introduce the topologically perturbed version of the state equation \eqref{eq:state_intro}. Let $q\in M\setminus \overline{\Omega}$, let $\delta>0$ small enough such that $U_q:=\exp_q(B_\delta(0)) \Subset M \setminus \overbar \Omega$ and set $\Omega_\eps := \Omega\cup \omega_\eps$. We denote by $u_\eps \in H^1(M)$ the solution of \eqref{eq:state_intro} with $\Omega = \Omega_\eps$, that is, 
\ben\label{eq:state_per}
\int_M \beta_{\Omega_\eps} \nabla^M u_\eps \cdot \nabla^M \varphi + \gamma_{\Omega_\eps} u_\eps \varphi \;dx = \int_M f_{\Omega_\eps}\varphi\;dx \quad \text{ for all } \varphi \in H^1(M). 
\een
\subsection{Preliminaries}
We make a few observations which will be helpful in the derivation of the topological derivative in the next section. For an injective matrix $A\in \VR^{n\times m}$ we define the pseudoinverse by $A^\dagger := (A^\top A)^{-1}A^\top$. For $x \in \VR^{d-1}$ and $\eps \geq0$ small we introduce the following notation:
\begin{align}
\Phi(x) :=& (\exp_q \circ E)(x), \label{eq_defPhi} \\
g_\eps(x) :=&  \sqrt{\mbox{det}( \partial \Phi(\eps x) ^\top \partial \Phi (\eps x) )}, \label{eq:g_eps} \\
\Psi_\eps(x):=& (\partial \Phi( \eps x ))^\dagger \in \VR^{d-1 \times d}, \\
A_\eps(x) :=& \Psi_\eps(x) \Psi_\eps(x)^\top \in \VR^{d-1 \times d-1}.\label{eq_defA}
\end{align}
Note that $\partial \Phi(\eps x) = \partial \exp_q(E(\eps x))(V) \in \VR^{d\times d-1}$ and $\partial T_\eps(x) = \eps \partial \Phi(\eps x)$.
Note that, since the manifold $M$ is assumed to be smooth, also $\exp_q$ is smooth on $B_\delta(0)$ and thus $\partial \Phi$ is well-defined in $B_\delta(0)$.

We collect the following properties of the transformation between the neighborhood $U_q$ of $q \in M$ and a subset of $\VR^{d-1}$ by the mappings $\Phi$ and $T_\eps$.
\begin{lemma} \label{lem_trafos}
    Let $T_\eps$ be as defined in \eqref{eq_defTeps} and $\Phi$, $g_\eps$, $\Psi_\eps$, $A_\eps$ as defined in \eqref{eq_defPhi}--\eqref{eq_defA}.
    \begin{enumerate}
     \item[a)]
    For $v \in H^1(M)$, it holds
    \begin{align}
        (\nabla^M v ) (q)&= ( \partial \Phi(0)^\dagger)^\top \nabla (v \circ \Phi)(0) = V \nabla (v\circ \Phi)(0), \label{eq_trafoPsi}
    \end{align}
     \item[b)]
    For $v \in H^1(M)$, it holds
    \begin{align}
        (\nabla^M v ) \circ T_\eps &= ( \partial T_\eps^\dagger)^\top \nabla (v \circ T_\eps) = \eps^{-1} \Psi_\eps^\top \nabla (v\circ T_\eps).
    \end{align}
    \item[c)] For $U \subset U_q \subset M$ and $f \in L_1(U)$, it holds
        \begin{align}
            \int_U f(x) \; dx =& \int_{T_\eps^{-1}(U)} \eps^{d-1} g_\eps(x) (f\circ T_\eps) (x)  \; dx .
        \end{align}
    \item[d)]
        For $U \subset U_q \subset M$ and $u, v \in H^1(M)$, it holds
        \begin{align}
            \int_U \nabla^M u \cdot \nabla^M v \; dx = \int_{T_\eps^{-1}(U)} \eps^{d-1} g_\eps(x) A_\eps(x) \nabla (u\circ T_\eps) \cdot \nabla (v \circ T_\eps) \; dx .
        \end{align}
    \end{enumerate}
\end{lemma}
\begin{proof}
    Point (a) follows since $q = \Phi(0)$, $\partial \Phi(0) = V$ and $V^\top V = I_{d-1}$. The result of (b) follows straightforwardly because $(\partial T_\eps^\dagger)^\top = \eps^{-1} (\partial \Phi(\eps x)^\dagger)^\top = \eps^{-1} \Psi_\eps(x)^\top$. Part (c) follows since $\sqrt{\mbox{det}( \partial T_\eps^\top \partial T_\eps )} = \sqrt{\mbox{det}( \eps^2 I_{d-1} \partial \Phi(\eps x)^\top \partial \Phi (\eps x) )} = \eps^{d-1} \sqrt{\mbox{det}( \partial \Phi(\eps x)^\top \partial \Phi (\eps x) )}$. Part (d) follows by combining parts (b) and (c).
\end{proof}

\begin{lemma} \label{lem_Ag_bounded}
    There exist constants $\underline c, \overbar c >0$ and $\tilde \eps > 0$ such that for all $\eps \in (0,\tilde{\eps})$, all $x \in T_\eps^{-1}(U_q)$ and all $v \in \VR^{d-1}$ it holds that
    \begin{align}
        \underline c  \leq& g_\eps(x) \leq \overbar c, \\
        \underline c |v| \leq& | \Psi_\eps(x)^\top v | \leq \overbar c |v|, \label{eq_PsiBounded}\\
        \underline c |v| \leq& | A_\eps(x) v | \leq \overbar c |v|. \label{eq_Abounded}
    \end{align}
\end{lemma}
\begin{proof}
    Since $\partial \Phi(0) = V$ and $V^\top V = I_{d-1}$, it holds for $\eps = 0$ that $g_0(x) = \sqrt{\mbox{det}(\partial \Phi(0)^\top \partial \Phi(0))} = 1$ independent of $x$. Since $\exp_q$ is smooth, there exists a constant $\tilde \eps$ such that $g_\eps(x)$ remains bounded and positive for all $\eps \in (0, \tilde \eps)$. Similarly, the smoothness of $\Phi$ also yields \eqref{eq_PsiBounded} and \eqref{eq_Abounded} since we have that $\Psi_0(x) =V^\top$ and $A_0(x) = V^\top V = I_{d-1}$.
\end{proof}

We further need the following well-known result: 
\begin{lemma}\label{L:geodesic_ball}
    We have 
    \ben\label{eq:expansion_omega_eps}
    |\omega_\eps| = |\omega| \eps^{d-1} + o(\eps^{d-1}), 
    \een
    where  $|\omega|:= \text{vol}_{d-1}(\omega)$ denotes the $d-1$ dimensional Lebesgue measure of $\omega$. 
\end{lemma}
\begin{proof}
    Since for all small $\eps \ge 0$ the map $x\mapsto  \exp_q(\eps Vx)$ is a diffeomorphism from $\omega$ onto $\omega_\eps$ we compute:
\begin{equation}
    |\omega_\eps| = \eps^{d-1}\int_\omega g_\eps(x)\;dx
\end{equation}
with $g_\eps$ defined as in \eqref{eq:g_eps}. In view of $g_0(x) = 1$ for all $x\in \omega$ the result follows from a Taylor expansion of $\eps \mapsto g_\eps(x)$ around $\eps =0$. 
\end{proof}

\begin{remark}
    For the case $\omega = B_1(0)$, the more general situation of Riemannian manifolds of arbitrary dimension is considered in \cite[Thm. 3.1]{a_Gr_1974a}. There, also the explicit expression of higher order terms, that means, the terms corresponding to $o(\eps^{d-1})$ in \eqref{eq:expansion_omega_eps}, are derived.
\end{remark}

From now on, we will consider only small values of $\eps \in [0, \tilde \eps)$ with $0 <\tilde \eps < \overbar \delta$ according to Lemma~\ref{lem_Ag_bounded}.

\section{Derivation of the topological derivative}\label{sec:topo_deriv}
In this section we consider the surface topology optimisation problem \eqref{eq:J_intro}--\eqref{eq:state_intro} with a tracking-type cost function $J$. More precisely, we consider the problem
\begin{align} \label{eq_J_alpha12}
\underset{(\Omega,u)}{\mbox{min }} J(u) :=& \alpha_1 \int_M |u-u_d|^2 \; \mbox dx +  \alpha_2 \int_M |\nabla^M(u-u_d)|^2 \; \mbox dx \\
\mbox{s. t. } u \in H^1(M): 
\int_M \beta_{\Omega}& \nabla^M u \cdot \nabla^M v + \gamma_\Omega uv \;dx = \int_M f_\Omega v\;dx \qquad \mbox{for all } v \in H^1(M). \label{eq_PDEConst}
\end{align}
where $u_d \in H^1(M)$ and $\alpha_1, \alpha_2 \geq 0$. Here, $\Omega$ is sought in a set of admissible open subsets $\mathcal A$, which is a subset of the power set $\mathcal P(M)$ of $M$, $\mathcal A \subset  \mathcal P(M)$. The coefficients $\beta_\Omega$, $\gamma_\Omega$ and $f_\Omega$ are as defined in \eqref{eq_defBetaOmega}. The adjoint equation associated to problem \eqref{eq_J_alpha12}--\eqref{eq_PDEConst} is to find $p_0 \in H^1(M)$ such that
\ben \label{eq_AdjEqn}
        \int_M \beta_{\Omega} \nabla^M p_0 \cdot \nabla^M v + \gamma_{\Omega} p_0 v \;dx = - 2\alpha_1 \int_M (u_0 - u_d) v \;dx - 2\alpha_2 \int_M \nabla^M ( u_0 -  u_d) \cdot \nabla^M v \;dx 
    \een
    for all $v \in H^1(M)$.

Recall the notation $\Cj(\Omega):= J(u(\Omega))$ for the reduced cost function where $u(\Omega)$ is the unique solution to \eqref{eq_PDEConst} for a given admissible set $\Omega\subset M$. Using the results of the previous section we derive the first order topological derivative of 
$\Cj$. 
\begin{definition}
    Let $\Cj$ be a shape function defined on a subset of $\Cp(M)$. Let $ \omega\subset \VR^{d-1}$ be an inclusion containing the origin and define $\omega_\eps:=\omega_\eps(q) := \exp_q(E(\eps \omega)) \subset M$. Let $\Omega\subset M$ be open with respect to the subspace topology.  We define the topological derivative of $\Cj$ at $\Omega$ with respect to the inclusion $\omega$ by 
\ben \label{eq_defTD}
d\Cj(\Omega)(q):= \left\{\begin{array}{ll}
        \lim_{\eps\searrow0} \frac{\Cj(\Omega\setminus \overline{ \omega}_\eps(q))- \Cj(\Omega)}{|\omega_\eps(q)|} & \text{ for } q\in \Omega,\\  
        \lim_{\eps\searrow0} \frac{\Cj(\Omega\cup \omega_\eps(q))- \Cj(\Omega)}{|\omega_\eps(q)|} & \text{ for } q\in M\setminus \overline \Omega.
\end{array}\right.
\een
\end{definition}

We will focus on the case where $q \in M \setminus \overline \Omega$ and remark that the other case where $q \in \Omega$ can be treated analgously by interchanging the roles of $\beta_1, \gamma_1, f_1$ and $\beta_2, \gamma_2, f_2$, respectively. For $q \in M \setminus \overline \Omega$, let $\Omega_\eps := \Omega \cup \omega_\eps$ denote the perturbed set.

The goal of this section is to prove the following result.
\begin{theorem} \label{thm_TD}
    Let $\omega\subset \VR^{d-1}$ be a bounded connected domain containing the origin and let an open set $\Omega\subset M$ and $q \in M \setminus \overbar \Omega$ be given. Let $u_0\in H^1(M)$ denote the solution to the state equation \eqref{eq_PDEConst} and $p_0\in H^1(M)$ the solution to the adjoint equation \eqref{eq_AdjEqn}. Assume further that there exists $\bar \delta > 0$ such that $u_0, p_0 \in C^1(B_{\bar \delta}(q) \cap M)$.
        
    Then the topological derivative of $\Cj$ defined by \eqref{eq_J_alpha12}--\eqref{eq_PDEConst} at a point $q \in M \setminus \overbar \Omega$ is given by 
    \ben
    d\Cj(\Omega)(q) = \partial_{\ell}G(0,u_0,p_0) + R(u_0,p_0) ,
    \een
    where
    \begin{align}
        \partial_{\ell}G(0, u_0, p_0) & =  (\beta_1-\beta_2) \nabla^M u_0(q)\cdot \nabla^M p_0(q) + (\gamma_1-\gamma_2) u_0(q) p_0(q) - (f_1-f_2) p_0(q) \\
        R(u_0,p_0)  & = (\beta_1-\beta_2) \frac{1}{|\omega|} \int_\omega V^\top \nabla^M u_0(q)\cdot \nabla Q(x)\;dx
    \end{align}
    and $Q\in \dot{BL}(\VR^{d-1})$ is the solution to
    \ben
    \int_{\VR^{d-1}} \beta_\omega  \nabla Q(x)\cdot \nabla v(x)\;dx =- (\beta_1-\beta_2) \int_\omega V^\top \nabla^M p_0(q)\cdot \nabla v(x)\;dx  - \alpha_2 \int_{\VR^{d-1}} \nabla K(x) \cdot \nabla v(x)\; dx,
    \een
    for all $v \in \dot{BL}(\VR^{d-1})$. Here, $K\in \dot{BL}(\VR^{d-1})$ is the solution to
    \ben
        \int_{\VR^{d-1}} \beta_\omega \nabla K(x)\cdot \nabla v(x)\;dx = -(\beta_1-\beta_2)\int_\omega V^\top \nabla^M u_0(q)\cdot \nabla v(x)\;dx \quad \text{ for all } v \in \dot{BL}(\VR^{d-1}).
    \een
\end{theorem} 

\begin{remark}
    We will see in Section \ref{sec_Qexpl} that, for the case $\alpha_2=0$ and $\omega = B_1(0)$ the unit ball, the function $Q$ can be computed explicitly and we get a closed form for the topological derivative which is independent of the choice of the basis $V$.
\end{remark}

In order to prove Theorem \ref{thm_TD} we apply the averaged adjoint framework introduced in \cite{Sturm2019} to the setting $X=Y=H^1(M)$, $\ell(\eps):= |\omega_\eps|$ and 
\ben \label{eq_defG}
\begin{split}
 G(\eps,\varphi,\psi) =& \alpha_1 \int_M |\varphi - u_d|^2 \; \mbox dx  +\alpha_2 \int_M |\nabla^M(\varphi - u_d)|^2 \; \mbox dx  \\
&+ \int_M \beta_{\Omega_\eps} \nabla^M \varphi \cdot \nabla^M\psi + \gamma_{\Omega_\eps} \varphi \psi - f_{\Omega_\eps} \psi\;dx.
\end{split}
\een

\subsection{Variation of the state}
In this section, we examine the difference between the solution $u_\eps$ of the perturbed state equation \eqref{eq:state_per} with $\eps > 0$ and the solution $u_0$ to the unperturbed state equation \eqref{eq_PDEConst}.

\begin{lemma}\label{lem:pert_state_apriori}
There is a constant $C>0$ such that for all $\eps \in (0, \tilde \eps)$,
\ben \label{eq_pert_state_apriori}
\|u_\eps - u_0\|_{H^1(M)} \le C \eps^{(d-1)/2}.
\een
\end{lemma}
\begin{proof}
Subtracting \eqref{eq:state_per} with $\eps = 0$ from that same equation with $\eps >0$, we obtain
\begin{align} \begin{aligned}\label{eq:state_pert_diff}
    \int_M\beta_{\Omega_\eps} \nabla^M (u_\eps-u_0) \cdot \nabla^M v   +& \gamma_{\Omega_\eps} (u_\eps- u_0) v \;dx=\int_{\omega_\eps}(f_1 -f_2)v \; dx \\ 
    & - (\beta_1 - \beta_2)\int_{\omega_\eps} \nabla^M u_0\cdot \nabla^M v\;dx 
    - (\gamma_1 - \gamma_2)\int_{\omega_\eps}  u_0  v\;dx 
\end{aligned}\end{align}
for all $v \in H^1(M)$. Hence testing with $v = u_\eps - u_0$, using the ellipticity of the left hand side, H\"older's inequality and that $u_0$ is continuously differentiable near $q$ yield
\ben
\|u_\eps - u_0\|_{H^1(M)} \le C \sqrt{|\omega_\eps|} \left( \|u_0\|_{C(B_{\bar \delta}(q)\cap M)} + \|\nabla^M u_0\|_{C(B_{\bar \delta}(q)\cap M)^d} \right),
\een
where $\overbar\delta >0$ is sufficiently small and $B_{\overbar\delta}(q)$ denotes the open ball in $\VR^d$ of radius 
$\overbar{\delta}$ centered at $q$. Now the result follows from $|\omega_\eps| = |\omega|\eps^{d-1} + o(\eps^{d-1})$ (see Lemma~\ref{L:geodesic_ball}). 
\end{proof}

In the following, we denote by $R: H^1(B) \rightarrow H^1(\VR^{d-1})$ the standard continuous Sobolev extension operator. 
\begin{definition} \label{def_Keps}
    For $\eps \in [0, \tilde \eps)$ we define the extension $\tilde u_\eps := R(u_\eps \circ \exp_q \circ E)$ and for $\eps \in (0, \tilde \eps)$ we define the variation of $u_\eps$ by 
\ben
K_\eps(x) := \left( \frac{\tilde u_\eps - \tilde u_0}{\eps} \right)(\eps x), \qquad x \in \VR^{d-1}.
\een
Notice that $K_\eps \in \dot{BL}(\VR^{d-1})$. 
\end{definition}
By changing variables in \eqref{eq_pert_state_apriori} and exploiting the boundedness stated in Lemma \ref{lem_Ag_bounded}, we obtain the following result:
\begin{corollary} \label{cor_Keps_bounded}
    There exists a constant $C>0$ such that for all $\eps \in (0, \tilde{\eps})$ it holds
    \begin{align}
        \int_{\VR^{d-1}} (\eps K_\eps)^2 + |\nabla K_\eps|^2 \, dx \leq C.
    \end{align}
\end{corollary}
\begin{proof}
    From \eqref{eq_pert_state_apriori} it follows that $\|u_\eps - u_0\|^2_{H^1(U_q)} \leq \|u_\eps - u_0\|_{H^1(M)}^2 \leq C \eps^{d-1}$ and thus, noting that $T_\eps^{-1}(U_q) = \eps^{-1}B$ and using Lemma \ref{lem_trafos}, a change of variable yields
    \begin{align}
        \int_{\eps^{-1}B} \eps^{d-1}g_\eps | \Psi_\eps^\top \nabla K_\eps|^2 + \eps^{d-1}g_\eps (\eps K_\eps)^2 \; dx \leq C\eps^{d-1}.
    \end{align}
    Dividing by $\eps^{d-1}$, Lemma \ref{lem_Ag_bounded} yields the assertion.

\end{proof}

The following result will be crucial for analysing the variation of the averaged adjoint states in Section \ref{sec_anaAdj}.
\begin{lemma} \label{thm_Keps_K}
    We have
    \begin{align}
        \nabla K_\eps \tow& \nabla K \quad \mbox{ in } L_2(\VR^{d-1})^{d-1},\\
        \eps K_\eps \tow& 0 \quad \, \quad \mbox{ in } L_2(\VR^{d-1}) , \label{eq_Keps_tow_0}
    \end{align}
    where $K \in  \dot{BL}(\VR^{d-1})$ is the unique solution to
    \ben\label{eq:limit_Keps_1}
        \int_{\VR^{d-1}} \beta_\omega \nabla K\cdot \nabla v\;dx = -(\beta_1-\beta_2)\int_\omega \nabla \tilde u_0(0)\cdot \nabla v\;dx \quad \text{ for all } v \in \dot{BL}(\VR^{d-1}),
    \een
where $\dot{BL}(\VR^d) := BL(\VR^d)/\VR$ is the so-called \emph{Beppo-Levi space} \cite{a_DELI_1955a,a_ORSU_2012a} and $BL(\VR^d) :=  \{u\in H^1_{\text{loc}}(\VR^d):\; \nabla u \in L_2(\VR^d)^d\}$ with $/\VR$ meaning that we quotient out constants. The norm on this space is defined by $\|[u]\|_{ \dot{BL}_p(\VR^d) } := \|\nabla u\|_{ L_p(\VR^d)^d}, u\in [u]$.
\end{lemma}

\begin{proof}
The proof is similar to the proof of \cite[Thm. 4.14]{Sturm2019}.
Recall that by construction the set $B_\delta(0)  \subset T_q M$ is chosen such that $\exp_q(B_\delta(0)) = U_q \Subset M \setminus \overbar \Omega$. Recall the set $B \subset \VR^{d-1}$ satisfying $E(B) = B_\delta(0)$, see also Figure \ref{fig_trafos}. Let $\tilde{\eps} > 0$ according to Lemma \ref{lem_Ag_bounded} and $\eps \in ( 0, \tilde{\eps})$ and $\bar v \in H^1_0({\tilde{\eps}}^{-1}B)$. For $\eps \in (0, \tilde \eps)$, let $v := \eps \bar v \circ T_\eps^{-1}$ and note that $v \in H^1(M)$.
    Testing \eqref{eq:state_pert_diff} with test functions $v$ of this form, a change of variables yields that
    \begin{align*}  
        \int_{{\tilde{\eps}}^{-1}B}  \beta_\omega g_\eps A_\eps& \nabla K_\eps \cdot \nabla \bar v \; \mbox dx
        + \int_{{\tilde{\eps}}^{-1}B} \gamma_\omega g_\eps \eps^2 K_\eps \bar v\; \mbox dx 
        = \int_\omega \eps g_\eps (f_1-f_2) \bar v \; dx\\
        &- (\beta_1 - \beta_2) \int_{\omega} g_\eps \nabla^M u_0(T_\eps(x)) \cdot (\Psi_\eps^\top \nabla \bar  v)\; \mbox dx - (\gamma_1 - \gamma_2) \int_{\omega}\eps g_\eps u_0(T_\eps(x)) \bar v \; \mbox dx
    \end{align*}
    for all $\bar v \in  H^1_0({\tilde{\eps}}^{-1}B)$. Defining
    \begin{align}
        r_1(\eps, \bar v) &:= - \int_{{\tilde{\eps}}^{-1}B} \gamma_\omega g_\eps \eps K_\eps \bar v\; \mbox dx, \\
        r_2(\eps, \bar v) &:= \int_\omega g_\eps (f_1-f_2) \bar v \; dx, \\
        r_3(\eps, \bar v) &:= - (\gamma_1 - \gamma_2) \int_{\omega} g_\eps u_0(T_\eps(x)) \bar v \; \mbox dx,
    \end{align}
    we get by rearranging
    \begin{align} \begin{aligned} \label{eq_Keps_rearr}
     \int_{{\tilde{\eps}}^{-1}B}  \beta_\omega g_\eps A_\eps& \nabla K_\eps \cdot \nabla \bar v \; \mbox dx 
     + (\beta_1 - \beta_2) \int_{\omega}g_\eps \nabla^M u_0(T_\eps(x)) \cdot (\Psi_\eps^\top \nabla \bar  v)\; \mbox dx
     =\eps \, (r_1(\eps, \bar v)+ r_2(\eps, \bar v)+ r_3(\eps, \bar v)).
    \end{aligned} \end{align}
    Using Lemma \ref{lem_Ag_bounded}, Corollary \ref{cor_Keps_bounded} and the fact that $u_0 \in C(B_{\bar \delta}(q) )$ for some $\bar \delta > \tilde \eps > 0$, H\"older's inequality yields the boundedness of the terms $r_i(\eps, \bar v)$, $i = 1,2,3$ in $L_2(\VR^{d-1})$:
    \begin{align*}
        |r_1(\eps, \bar v)| &\leq C \| \eps K_\eps \|_{L_2(\VR^{d-1})}\| \bar v \|_{L_2(\VR^{d-1})} \leq C \| \bar v \|_{L_2(\VR^{d-1})}, \\
        |r_2(\eps, \bar v)| & \leq C \| \bar v \|_{L_2(\omega)} \leq C \| \bar v \|_{L_2(\VR^{d-1})}, \\
        |r_3(\eps, \bar v)| &\leq C \| u_0 \|_{C(B_{\bar \delta}(q) \cap M)} \| \bar v \|_{L_2(\omega)} \leq C \| \bar v \|_{L_2(\VR^{d-1})},
    \end{align*}
    for all $\bar v \in H^1({\tilde{\eps}}^{-1} B)$ and $\eps  \in (0, \tilde{\eps})$. Since the family $(K_\eps)_\eps$ is bounded in the Hilbert space $\dot{BL}(\VR^{d-1})$ due to Corollary \ref{cor_Keps_bounded}, we can find for every null sequence $(\eps_n)_n$ a subsequence $(\eps_{n_k})$ and an element $\overbar K \in \dot{BL}(\VR^{d-1})$ such that for the corresponding sequence $(K_{\eps_{n_k}})_k$ it holds $\nabla K_{\eps_{n_k}} \tow  \nabla \overbar K$ in $L_2(\VR^{d-1})^{d-1}$ as $k \to \infty$. Thus, setting $\eps = \eps_{n_k}$ in \eqref{eq_Keps_rearr}, we can pass to the limit $k \rightarrow \infty$ to obtain
    \begin{align} \label{eq_Kbar}
        \int_{{\tilde{\eps}}^{-1}B} \beta_\omega \nabla \overbar K \cdot \nabla \bar v \; \mbox dx= -(\beta_1 - \beta_2) \int_\omega \nabla^M u_0(q) \cdot  V \nabla \bar v \; \mbox dx 
    \end{align}
    for all $\bar v \in H^1_0({\tilde{\eps}}^{-1}B) \subset H^1(\VR^{d-1})$. Here we used that, for $\eps =0$, we have $A_0 = I_{d-1}$, $g_0=1$ and $\Psi_0^\top = (V^\dagger)^\top = V$.
    Since $\tilde \eps>0$ was arbitrary, we can replace ${\tilde{\eps}}^{-1}B$ by $\VR^{d-1}$ in \eqref{eq_Kbar}. Noting that 
    \begin{align*}
        \nabla^M u_0(q) = V \nabla \tilde u_0 (0),
    \end{align*}
    and $V^\top V = I_{d-1}$, we see that \eqref{eq_Kbar} then coincides with \eqref{eq:limit_Keps_1}. It follows immediately by the lemma of Lax-Milgram that problem \eqref{eq:limit_Keps_1} admits a unique solution. Thus, we conclude that $\overbar K = K$ and $\nabla K_\eps \tow \nabla K$ in $L_2(\VR^{d-1})^{d-1}$.

    The second statement \eqref{eq_Keps_tow_0} can be shown in the same way as it was done in \cite[Thm. 4.14]{Sturm2019}.
\end{proof}

\subsection{Variation of the averaged adjoint variable} \label{sec_anaAdj}

We proceed by studying the averaged adjoint equation corresponding to problem \eqref{eq_J_alpha12}--\eqref{eq_PDEConst}, which reads: 
\ben \label{eq_avgedAdjEqnG}
\mbox{find } p_\eps \in H^1(M): \int_0^1 \partial_u G(\eps,su_\eps + (1-s)u_0,p_\eps)(v)\;ds =0 \quad \text{ for all } v \in H^1(M),
\een
or equivalently: find $p_\eps \in H^1(M)$ such that
\ben \label{eq_avgedAdjEqn}
\int_M \beta_{\Omega_\eps} \nabla^M p_\eps \cdot \nabla^M v + \gamma_{\Omega_\eps} p_\eps v\;dx = - \alpha_1 \int_M (u_\eps + u_0 - 2 u_d) v\;dx - \alpha_2 \int_M \nabla^M (u_\eps + u_0 - 2 u_d) \cdot \nabla^M v \;dx 
\een
for all $v \in H^1(M)$. We can state a similar result to Lemma \ref{lem:pert_state_apriori}.

\begin{lemma} \label{lem:pert_avgdadj_apriori}
    There is a constant $C>0$ such that for all $\eps \in (0, \tilde \eps)$,
    \ben \label{eq_pert_avgdadj_apriori}
        \|p_\eps - p_0\|_{H^1(M)} \leq C \eps^{(d-1)/2}.
    \een
\end{lemma}
\begin{proof}
    Subtracting \eqref{eq_avgedAdjEqn} with $\eps =0$ from that same equation with $\eps > 0$, we get
    \begin{align} \label{eq_peps_p0}
        \begin{aligned}
        \int_M & \beta_{\Omega_\eps} \nabla^M (p_\eps-p_0) \cdot \nabla^M v + \gamma_{\Omega_\eps} (p_\eps-p_0) v\; \mbox dx = - \alpha_1 \int_M (u_\eps - u_0) v\; \mbox dx \\
        &- \alpha_2 \int_M \nabla^M (u_\eps - u_0) \cdot \nabla^M v\; \mbox dx 
        - (\beta_1 - \beta_2) \int_{\omega_\eps} \nabla^M p_0 \cdot \nabla^M  v\; \mbox dx - (\gamma_1 - \gamma_2) \int_{\omega_\eps}p_0 v \; \mbox dx.
        \end{aligned}
    \end{align}
    Testing with $v = p_\eps - p_0$ and using the ellipticity with respect to $H^1(M)$ of the left hand side, Hölder's inequality and the fact that $p_0$ is continuously differentiable near $q$, we arrive at
    \begin{align*}  
        \| p_\eps - p_0 \|_{H^1(M)} \leq C \left( \| u_\eps - u_0 \|_{H^1(M)} + |\omega_\eps|^{1/2} \left(\|p_0\|_{C(B_{\bar \delta}(q)\cap M)} + \|\nabla^M p_0\|_{C(B_{\bar \delta}(q)\cap M)^d} \right) \right),
    \end{align*}
    where $\bar \delta >0$ is sufficiently small and $B_{\bar \delta}(q)$ denotes the open ball in $\VR^d$ of radius 
    $\bar \delta$ centered at $q$. Using Lemma \ref{L:geodesic_ball} and Lemma \ref{lem:pert_state_apriori}, we obtain the result.
\end{proof}

\begin{definition}
    As in Definition~\ref{def_Keps} we define the extension $\tilde p_\eps := R(p_\eps \circ \exp_q \circ E)$ and define the variation of $p_\eps$ by 
\ben
Q_\eps(x) := \left( \frac{\tilde p_\eps - \tilde p_0}{\eps} \right)(\eps x),\qquad x \in \VR^{d-1}.
\een
Again notice that $Q_\eps \in \dot{BL}(\VR^{d-1})$. 
\end{definition}

Performing the change of variables $x = T_\eps(y)$ in  \eqref{eq_pert_avgdadj_apriori} and exploiting the boundedness of $g_\eps$ and $\Psi_\eps$ according to Lemma \ref{lem_Ag_bounded}, the following result can be shown in the exact same way as in Corollary \ref{cor_Keps_bounded}.
\begin{corollary} \label{cor_Qeps_bounded}
    There is a constant $C>0$ such that for all $\eps \in (0, \tilde \eps)$, it holds
    \ben
    \int_{\VR^{d-1}} (\eps Q_\eps)^2 +| \nabla  Q_\eps |^2 \leq C.
    \een
\end{corollary}

The following result is similar to the result of Lemma \ref{thm_Keps_K} and will be crucial for the rigorous justification of the topological derivative in Section \ref{sec_TD}.
\begin{lemma} \label{lem_convQ}
We have
\begin{align}
\nabla Q_\eps &\tow \nabla Q \quad \text{ weakly in } L_2(\VR^{d-1})^{d-1}, \\
\eps Q_\eps &\tow 0 \qquad \text{ weakly in } L_2(\VR^{d-1}), \label{eq_Qeps_tow_0}
\end{align}
where $Q\in \dot{BL}(\VR^{d-1})$ denotes the unique solution to 
\begin{align} \label{eq_defQ}
    \int_{\VR^{d-1}} \beta_\omega \nabla Q \cdot \nabla v \; \mbox dx= -(\beta_1 - \beta_2) \int_\omega \nabla \tilde p_0(0) \cdot  \nabla v \; \mbox dx - \alpha_2 \int_{\VR^{d-1}} \nabla K \cdot \nabla v \; \mbox dx
\end{align}
for all $v \in \dot{BL}(\VR^{d-1})$.
\black

\end{lemma}
\begin{proof}
    We proceed in a similar way as in the proof of Lemma \ref{thm_Keps_K}.
   
    We test the equation which is fulfilled by the variation $p_\eps - p_0$ \eqref{eq_peps_p0} with test functions of the form $v = \eps \bar v \circ T_\eps^{-1}$ where $\eps \in (0, \tilde \eps)$ and $\bar v \in H^1_0(\tilde \eps^{-1} B)$. Then, a change of variables yields (similar to the proof of Lemma \ref{thm_Keps_K}) that
    \begin{align*}  
        \int_{{\tilde{\eps}}^{-1}B}  \beta_\omega g_\eps A_\eps(x)& \nabla Q_\eps \cdot \nabla \bar v \; \mbox dx
        + \int_{{\tilde{\eps}}^{-1}B} \gamma_\omega g_\eps \eps^2Q_\eps \bar v\; \mbox dx \\
        =& - \alpha_1 \int_{{\tilde{\eps}}^{-1}B} g_\eps \eps^2 K_\eps \bar v\; \mbox dx  - \alpha_2 \int_{{\tilde{\eps}}^{-1}B}g_\eps A_\eps(x)  \nabla K_\eps \cdot \nabla \bar v\; \mbox dx \\
        &- (\beta_1 - \beta_2) \int_{\omega} \nabla^M p_0(T_\eps(x)) \cdot (\Psi_\eps^\top \nabla \bar  v)\; \mbox dx - (\gamma_1 - \gamma_2) \int_{\omega}\eps g_\eps p_0(T_\eps(x)) \bar v \; \mbox dx
    \end{align*}
    for all $\bar v \in  H^1_0({\tilde{\eps}}^{-1}B)$. Defining
    \begin{align}
        r_1(\eps, \bar v) &:= - \int_{{\tilde{\eps}}^{-1}B} \gamma_\omega g_\eps \eps Q_\eps \bar v\; \mbox dx \\
        r_2(\eps, \bar v) &:=  - \alpha_1 \int_{{\tilde{\eps}}^{-1}B} g_\eps \eps K_\eps \bar v\; \mbox dx \\
        r_3(\eps, \bar v) &:= - (\gamma_1 - \gamma_2) \int_{\omega} g_\eps p_0(T_\eps(x)) \bar v \; \mbox dx
    \end{align}
    we get by rearranging
    \begin{align} \begin{aligned} \label{eq_Qeps_rearr}
     \int_{{\tilde{\eps}}^{-1}B}  \beta_\omega g_\eps A_\eps(x)& \nabla Q_\eps \cdot \nabla \bar v \; \mbox dx + \alpha_2 \int_{{\tilde{\eps}}^{-1}B}g_\eps A_\eps(x)  \nabla K_\eps \cdot \nabla \bar v\; \mbox dx \\
     &+ (\beta_1 - \beta_2) \int_{\omega} \nabla^M p_0(T_\eps(x)) \cdot (\Psi_\eps^\top \nabla \bar  v)\; \mbox dx
     =\eps \, (r_1(\eps, \bar v)+ r_2(\eps, \bar v)+r_3(\eps, \bar v)).
    \end{aligned} \end{align}
    Using Lemma \ref{lem_Ag_bounded}, Corollary \ref{cor_Qeps_bounded}, Corollary \ref{cor_Keps_bounded} and the fact that $p_0 \in C(B_{\bar \delta}(q) )$ for some $\bar \delta > \tilde \eps > 0$, again Hölder's inequality yields the boundedness of the terms $r_i(\eps, \bar v)$, $i = 1,2,3$ in $L_2(\VR^{d-1})$:
    \begin{align*}
        |r_1(\eps, \bar v)| &\leq C \| \eps Q_\eps \|_{L_2(\VR^{d-1})}\| \bar v \|_{L_2(\VR^{d-1})} \leq C \| \bar v \|_{L_2(\VR^{d-1})}, \\
        |r_2(\eps, \bar v)| &\leq C \| \eps K_\eps \|_{L_2(\VR^{d-1})}\| \bar v \|_{L_2(\VR^{d-1})} \leq C \| \bar v \|_{L_2(\VR^{d-1})}, \\
        |r_3(\eps, \bar v)| &\leq C \| p_0 \|_{C(B_{\bar \delta}(q))} \| \bar v \|_{L_2(\omega)} \leq C \| \bar v \|_{L_2(\VR^{d-1})},
    \end{align*}
    for all $\bar v \in H^1_0({\tilde{\eps}}^{-1} B)$ and $\eps  \in (0, \tilde{\eps})$. The family $(Q_\eps)_\eps$ is bounded in the Hilbert space $\dot{BL}(\VR^{d-1})$ due to Corollary \ref{cor_Qeps_bounded}. Therefore, for every null sequence $(\eps_n)_n$ there exists a subsequence $(\eps_{n_k})$ and an element $\overbar Q \in \dot{BL}(\VR^{d-1})$ such that the corresponding sequence $(\nabla Q_{\eps_{n_k}})_k$ converges weakly to that element $\overbar Q$, $\nabla Q_{\eps_{n_k}} \tow  \nabla \overbar Q$ in $L_2(\VR^{d-1})^{d-1}$ as $k\to \infty$. Thus, setting $\eps = \eps_{n_k}$ in \eqref{eq_Qeps_rearr} and noting that $\nabla K_\eps \tow \nabla K$ in $L_2(\VR^{d-1})^{d-1}$ according to Lemma \ref{thm_Keps_K}, we can pass to the limit $k \rightarrow \infty$ and obtain
    \begin{align} \label{eq_Qbar}
        \int_{{\tilde{\eps}}^{-1}B} \beta_\omega \nabla \overbar Q \cdot \nabla \bar v \; \mbox dx= -(\beta_1 - \beta_2) \int_\omega \nabla^M p_0(q) \cdot  (V^\dagger)^\top \nabla \bar v \; \mbox dx - \alpha_2 \int_{{\tilde{\eps}}^{-1}B} \nabla K \cdot \nabla \bar v \; \mbox dx
    \end{align}
    for all $\bar v \in H^1_0({\tilde{\eps}}^{-1}B) \subset H^1(\VR^{d-1})$. Since $\tilde \eps>0$ was arbitrary, we can replace ${\tilde{\eps}}^{-1}B$ by $\VR^{d-1}$ in \eqref{eq_Qbar}. Noting that $V^\top V = I_{d-1}$, $(V^\dagger)^\top = V$ and
    \begin{align*}
        \nabla^M p_0(q) = V \nabla \tilde p_0 (0),
    \end{align*}
    we see that \eqref{eq_Qbar} then coincides with \eqref{eq_defQ}. Since \eqref{eq_defQ} has a unique solution, we conclude that $\overbar Q = Q$ and $\nabla Q_\eps \tow \nabla Q$ in $L_2(\VR^{d-1})^{d-1}$.
  
    The second statement \eqref{eq_Qeps_tow_0} can be shown in the same way as it was done in \cite[Thm. 4.14]{Sturm2019}.

\end{proof}

\subsection{Topological derivative} \label{sec_TD}
Using the convergence behaviour of $Q_\eps$ stated in Lemma \ref{lem_convQ}, we can now derive the topological derivative of the surface PDE constrained topology optimisation problem \eqref{eq_J_alpha12}--\eqref{eq_PDEConst}. We use the approach introduced in \cite{Sturm2019}, see also \cite{a_GAST_2019a.13420v2}.

Recall the definition of the Lagrangian $G$ \eqref{eq_defG}. Note that, for any $\eps \in [0, \tilde \eps)$, the perturbed state equation \eqref{eq:state_per} and the averaged adjoint equation \eqref{eq_avgedAdjEqn} admit unique solutions $u_\eps \in H^1(M)$ and $p_\eps \in H^1(M)$, respectively. Further note that, for $q \in M \setminus \overbar \Omega$ and $\eps \in [0, \tilde \eps)$ it holds
\ben
    \mathcal J(\Omega \cup \omega_\eps(q)) = G(\eps, u_\eps, \psi)
\een
for any $\psi \in H^1(M)$ since $u_\eps$ solves \eqref{eq:state_per}. Thus, the topological derivative defined in \eqref{eq_defTD} can be rewritten for the problem at hand as
\ben
    d\mathcal J(\Omega)(q) = \underset{\eps\searrow 0}{\mbox{lim }} \frac{\mathcal J(\Omega \cup \omega_\eps(q)) - \mathcal J(\Omega)}{|\omega_\eps(q)|} = \underset{\eps\searrow 0}{\mbox{lim }} \frac{G(\eps, u_\eps, p_\eps) - G(0, u_0, p_0)}{|\omega_\eps(q)|}.
\een
The fundamental theorem of calculus yields for all $\eps \in (0, \tilde \eps)$ that
\ben
    G(\eps, u_\eps, p_\eps) = G(\eps, u_0, p_\eps) + \int_0^1 \partial_u G(\eps, s u_\eps + (1-s)u_0, p_\eps)(u_\eps - u_0) \; ds = G(\eps, u_0, p_\eps)
\een
since $p_\eps$ solves \eqref{eq_avgedAdjEqnG}. Thus, we have
\begin{align*}
G(\eps, u_\eps, p_\eps) - G(0, u_0, p_0) =& G(\eps, u_0, p_\eps) - G(0, u_0, p_0) \\
=& G(\eps, u_0, p_\eps) - G(\eps, u_0, p_0) + G(\eps, u_0, p_0) - G(0, u_0, p_0)
\end{align*}
and we obtain for the topological derivative
\ben
    d\mathcal J(\Omega)(q) = \partial_\ell G(0, u_0, p_0) + R(u_0, p_0)
\een
with
\begin{align}
        \partial_\ell G(0, u_0, p_0) :=&  \underset{\eps\searrow 0}{\mbox{lim }} \frac{G(\eps,u_0, p_0) - G(0,u_0, p_0)}{|\omega_\eps(q)|}, \label{eq_def_dlG} \\
        R(u_0, p_0) :=&  \underset{\eps\searrow 0}{\mbox{lim }} \frac{G(\eps,u_0, p_\eps) - G(\eps,u_0, p_0)}{|\omega_\eps(q)|}, \label{eq_defR}
\end{align}
if these limits exist.

Using that $u_0$ and $p_0$ are of class $C^1$ around $q$, it follows that $\partial_{\ell}G(0,u_0,p_0)$ in \eqref{eq_def_dlG} exists with
\begin{align}
    \partial_{\ell}G(0,u_0,p_0) & = (\beta_1-\beta_2)\nabla^M u_0(q)\cdot \nabla^M p_0(q) +(\gamma_1 - \gamma_2) u_0(q) p_0(q) - (f_1-f_2)p_0(q) \nonumber \\
                                & = (\beta_1-\beta_2)  \nabla \tilde u_0(0)\cdot \nabla \tilde p_0(0) + (\gamma_1 - \gamma_2) \tilde u_0(0) \tilde p_0(0)- (f_1-f_2)\tilde p_0(0).
\end{align}
Exploiting the convergence behaviour of $Q_\eps$ established in Lemma \ref{lem_convQ}, we can show the existence of the term $R(u_0, p_0)$ in \eqref{eq_defR}:

\begin{lemma}
We have 
\ben\label{eq:formula_R}
R(u_0,p_0) = (\beta_1-\beta_2)  \frac{1}{|\omega|}  \int_\omega \nabla \tilde u_0(0)\cdot \nabla Q\;dx 
\een
\end{lemma}
\begin{proof}
Using \eqref{eq_PDEConst} with $v = p_\eps - p_0$ and changing variables, we compute 
\ben
\begin{split}
    G(\eps,u_0,p_\eps) & - G(\eps,u_0,p_0)  = \int_M \beta_{\Omega_\eps} \nabla^M u_0 \cdot \nabla^M (p_\eps - p_0) + \gamma_{\Omega_\eps} u_0 (p_\eps - p_0) - f_{\Omega_\eps}(p_\eps - p_0) \; \mbox dx \\
    =&    \int_{\omega_\eps} (\beta_1 - \beta_2) \nabla^M u_0 \cdot \nabla^M (p_\eps - p_0) \; \mbox dx  + \int_{\omega_\eps} (\gamma_1 - \gamma_2) u_0 (p_\eps - p_0) \; \mbox dx  - \int_{\omega_\eps} (f_1 - f_2) (p_\eps - p_0) \; \mbox dx \\
    =& (\beta_1 - \beta_2) \eps^{d-1} \int_\omega g_\eps \nabla^M u_0(T_\eps(x) )\cdot \Psi_\eps^\top \nabla Q_\eps\; dx \\
    &+ (\gamma_1 - \gamma_2) \eps^{d-1} \int_\omega u_0( T_\eps(x) )  \eps Q_\eps \; \mbox dx - (f_1 - f_2) \eps^{d-1} \int_\omega \eps Q_\eps \; \mbox dx.
\end{split}
\een
Hence dividing by $|\omega_\eps|$ and passing to the limit $\eps \searrow 0$ yields \eqref{eq:formula_R}, where we used that $\eps^{d-1} / |\omega_\eps| \to 1/|\omega|$ as $\eps \to 0$ (cf. Lemma \ref{L:geodesic_ball}), $\Psi_0^\top = V$, \eqref{eq_trafoPsi} and $V^\top V = I$. 
\end{proof}

Hence, for $q \in M \setminus  \overline \Omega$, the topological derivative reads
\begin{align}
    dJ(\Omega)(q) =& \partial_\ell G(0, u_0, p_0) + R(u_0, p_0) \nonumber \\
    =& (\beta_1 - \beta_2) \frac{1}{|\omega|} \int_\omega \nabla \tilde u_0(0) \cdot \left(\nabla \tilde p_0(0) + \nabla Q \right) \; \mbox dx \\
    &+ (\gamma_1 - \gamma_2) \tilde u_0(0) \tilde p_0(0) - (f_1-f_2) \tilde p_0(0).
\end{align}
This finishes the proof of Theorem \ref{thm_TD}.  
\begin{flushright}
    $\blacksquare$ 
\end{flushright}

\subsection{Explicit determination of $Q$} \label{sec_Qexpl}
When $\alpha_2 = 0$ and $\omega =B_1(0)$ we can compute the solution $Q$ to problem \eqref{eq_defQ} explicitly by the ansatz $Q(x) = \sum_{i=1}^{d-1} P_i Q_{e_i}(x)$ with $P_i$ the components of $\nabla \tilde p_0(0) = (P_1,\dots, P_{d-1})^\top \in \VR^{d-1}$ and 
\begin{equation} \label{eq_def_Qei}
    Q_{e_i}(x) := \begin{cases}
                        a_i x_i =: Q_{e_i}^{in}(x) & \mbox{in }\omega,\\
                        a_i \frac{x_i}{|x|^{d-1}} =: Q_{e_i}^{out}(x) & \mbox{in }\VR^{d-1} \setminus \omega,
                    \end{cases}
\end{equation}
for $i \in \{1,\dots d-1\}$; see also \cite[Rem. 6.10]{AmstutzGangl2019} and \cite[Prop. 1]{a_HILA_2008a}. Here $Q_{e_i}$ should solve 
\ben\label{eq:Q_ei}
\int_{\VR^{d-1}} \beta_\omega \nabla Q_{e_i}\cdot \nabla \varphi = - (\beta_1-\beta_2)  \int_\omega e_i \cdot\nabla \varphi \;dx \quad \text{ for all } \varphi \in \dot{BL}(\VR^{d-1})
\een
with the $i$-th unit vector $e_i$. Problem \eqref{eq:Q_ei} can be rewritten in strong form as the transmission problem
\begin{subequations}
\begin{align}
    - \beta_2 \Delta Q^{out} &= 0 \qquad &&\mbox{in } \VR^{d-1} \setminus \omega, \label{eq_Q1ei_strong_out} \\
    - \beta_2 \Delta Q^{ in} &= (\beta_1-\beta_2) \mbox{div}(e_i + \nabla Q^{ in} )  \qquad &&\mbox{in } \omega, \label{eq_Q1ei_strong_in}\\
    \left( \beta_1 \nabla Q^{ in} - \beta_2 \nabla Q^{ out} \right) \cdot n_{out} &= -(\beta_1 - \beta_2)e_i \cdot n_{out} \qquad &&\mbox{on } \partial \omega, \label{eq_Q1ei_strong_transGrad}\\
    Q^{ in} &= Q^{ out} \qquad &&\mbox{on } \partial \omega, \label{eq_Q1ei_strong_trans}
\end{align}
\end{subequations}
where $n_{out}$ denotes the unit normal vector pointing out of $\omega$. We see immediately that $Q_{e_i}$ defined in \eqref{eq_def_Qei} satisfies \eqref{eq_Q1ei_strong_in} and \eqref{eq_Q1ei_strong_trans}. Also \eqref{eq_Q1ei_strong_out} is readily verified. Furthermore, it can be seen that with the choice $a_i = - \frac{\beta_1 - \beta_2}{\beta_1 + (d-2)\beta_2}$ also the transmission condition \eqref{eq_Q1ei_strong_transGrad} is satisfied. Note that the constants $a_i$ are independent of the index $i$. Thus,
\ben
    \nabla Q|_{\omega} = -\frac{\beta_1 - \beta_2}{\beta_1 + (d-2)\beta_2} \nabla \tilde p_0(0),
\een
and thus, for $d=3$ we have
\ben  \label{eq_TDprefinal}
    dJ(\Omega)(q) =  2 \beta_2 \frac{\beta_1-\beta_2}{\beta_1 + \beta_2} \nabla \tilde u_0(0) \cdot \nabla \tilde p_0(0) + (\gamma_1 - \gamma_2) \tilde u_0(0) \tilde p_0(0) - (f_1-f_2) \tilde p_0(0)
\een
for $q \in M \setminus \overbar \Omega$. Note that \eqref{eq_TDprefinal} has the same structure as the analogous formula for the case of PDEs posed on volumes, see \cite[Thm. 6.1]{Amstutz2006}.
A similar procedure is also possible for ellipse-shaped inclusions $\omega$.

It is readily verified that, for $q \in \Omega$ and $d=3$, the topological derivative reads
\ben   
    dJ(\Omega)(q) =  2 \beta_1 \frac{\beta_2-\beta_1}{\beta_2 + \beta_1} \nabla \tilde u_0(0) \cdot \nabla \tilde p_0(0) + (\gamma_2 - \gamma_1) \tilde u_0(0) \tilde p_0(0) - (f_2-f_1) \tilde p_0(0). 
\een
Using \eqref{eq_trafoPsi} and $V^\top V = I_{d-1}$, note that it holds that $\nabla^M u_0(q) \cdot \nabla^M p_0(q) = \nabla \tilde u_0(0) \cdot \nabla \tilde p_0(0)$. Summarizing, we have shown the following corollary.

\begin{corollary} \label{cor_TDspecialCase}
    Let $d=3$, $\omega = B_1(0)$ and $\alpha_2=0$. Then, the topological derivative of problem \eqref{eq_J_alpha12}--\eqref{eq_PDEConst} reads
    \ben \label{eq_TDfinal}
        d \Cj(\Omega)(q) = 2 \beta_2 \frac{\beta_1-\beta_2}{\beta_1 + \beta_2} \nabla^M u_0(q) \cdot \nabla^M p_0(q) + (\gamma_1 - \gamma_2) u_0(q) p_0(q) - (f_1-f_2) p_0(q)
    \een
    for $q \in M \setminus \overbar \Omega$, and 
    \ben \label{eq_TDfinal_qOmega}
        d \Cj(\Omega)(q) =2 \beta_1 \frac{\beta_2-\beta_1}{\beta_2 + \beta_1} \nabla^M u_0(q) \cdot \nabla^M p_0(q) + (\gamma_2 - \gamma_1) u_0(q) p_0(q) - (f_2-f_1) p_0(q). 
    \een
    for $q \in \Omega$.
\end{corollary}

\section{Numerical results}\label{sec:numerics}
In this section, we illustrate the use of the topological derivative derived in the previous section in a numerical topology optimization example posed on a sphere in three space dimensions. The sphere $M$ is interpreted as the surface of the planet earth and the data of the problem is chosen in such a way that the optimal shape $\Omega^* \subset M$ represents the major land masses of the planet.

\subsection{Problem setting} \label{sec_problemSetting}
We consider the problem of minimising the objective function \eqref{eq_J_alpha12} subject to the surface PDE problem given by \eqref{eq_PDEConst}. We choose the parameters $\alpha_1 = 1$ and $\alpha_2=0$, $\beta_1 = 10^4$, $\beta_2 = 10^{-3}$, $\gamma_1 = \gamma_2 =1$, $f_1 = 10^3 $, $f_2=0$. Thus, the problem reads
\begin{subequations} \label{eq_problemNumerics}
 \begin{align}
    \underset{(\Omega,u)}{\mbox{min}} &\, \int_M |u-u_d|^2 \; \mbox dx  \\
    \mbox{subject to }u \in H^1(M): 
\int_M \beta_{\Omega}& \nabla^M u \cdot \nabla^M v + uv \;dx = f_1\int_\Omega  v\;dx \qquad \mbox{for all } v \in H^1(M) \label{eq_stateNumerics}
\end{align}
\end{subequations}
with $\beta_\Omega(x) = \chi_\Omega(x) \beta_1 + \chi_{M \setminus \Omega}(x) \beta_2$.

In order to define a desired state $u_d$, we choose a reference shape $\Omega^{*}$, compute the corresponding solution to the surface PDE \eqref{eq_stateNumerics} $u^{*}$ and set $u_d :=  u^{*}$. Then, by construction, $\Omega^*$ is also the solution of problem \eqref{eq_problemNumerics}. The reference shape chosen for this numerical example is given by topographical data of the land masses of the earth, which we obtained from \cite{earthData}. The problem at hand can be interpreted as a steady state heat conduction problem where the land masses $\Omega$ have very high conductivity and the water regions $M \setminus \overbar \Omega$ very low conductivity. A heat source is supported on the land masses $\Omega$. 

\subsection{Optimization algorithm} \label{sec_optiAlgo}
We solve the problem by means of the level set algorithm introduced in \cite{a_AMAN_2006a}, which is based solely on the topological derivative. In \cite{a_AMAN_2006a}, the algorithm is introduced in the setting of topology optimization problems which are constrained by PDEs on volumes, however, the extension to surface PDE constraints is straightforward. The idea of the algorithm is to represent the design $\Omega \subset M$ by means of a level set function $\psi : M \rightarrow \VR$ as $\Omega = \{x \in M : \psi(x) < 0\}$. Introducing the so-called generalized topological derivative,
\begin{equation}
    g_\Omega(q) := \begin{cases}
                        -dJ(\Omega)(q) &q \in \Omega, \\
                        dJ(\Omega)(q) &q \in M \setminus \Omega,
                   \end{cases}
\end{equation}
it follows that a stationarity condition is given by 
\begin{equation}
    \psi(q) = g_\Omega(q) \mbox{ for all }q \in M \setminus \partial \Omega.
\end{equation}
The idea of the algorithm is to reach this condition by a spherical linear interpolation (SLERP) iteration on the unit sphere $\mathcal S$ of the Hilbert space $L_2(M)$. We start the algorithm with an initial design $\Omega_0$ and the corresponding level set function $\psi_0$, which we assume to be normalized, $ \|\psi_0\|_{L_2(M)}=1$. In iteration $k\geq0$ of the algorithm, let $\Omega_k$ the current shape, $\psi_k$ the corresponding level set function, and 
\benn
\theta_k = \mbox{arccos}\left( \left(\psi_k , \frac{g_{\Omega_k}}{\|g_{\Omega_k} \|_{L_2(M)}} \right)_{L_2(M)} \right)
\eenn
the angle between $\psi_k$ and $g_{\Omega_k}$ in an $L_2(M)$-sense. Then the new iterate $\psi_{k+1}$ is given by
\begin{equation} \label{eq_levelsetupdate}
    \psi_{k+1} = \frac{1}{\mbox{sin}(\theta_k)} \left( \mbox{sin}((1-\kappa_k)\theta_k) \, \psi_k + \mbox{sin}(\kappa_k\theta_k) \frac{g_{\Omega_k}}{\|g_{\Omega_k}\|_{L^2(M)}}\right).
\end{equation}
Here, $\kappa_k \in (0,1]$ is a line search parameter which is adapted in every iteration in order to achieve a sufficient descent of the objective function. Note that, by construction, it follows from $\|\psi_0\|_{L_2(M)} =1$ that $\|\psi_k\|_{L_2(M)}=1$ for all $k > 0$. For more details on the algorithm and its implementation, we refer the interested reader to \cite{a_AMAN_2006a}.

\subsection{Numerical experiments}
We now show numerical results obtained by applying the level set algorithm introduced in Section \ref{sec_optiAlgo} to the problem described in Section \ref{sec_problemSetting} using the topological derivative formulas \eqref{eq_TDfinal} for $q \in M \setminus \overline{\Omega}$ and \eqref{eq_TDfinal_qOmega} for $q \in \Omega$.

The surface $M$ is chosen as the unit sphere in three space dimensions, which we discretized into 161620 triangular surface elements with 80812 vertices, see Figure \ref{fig_TDs}(a). In order to determine the desired shape $\Omega^*$ representing the major land masses of the planet, we used the data obtained from \cite{earthData} to decide for every triangular surface element whether it should belong to land or water regions. This decision is made based on the position of the element's centroid. The left columns of Figures \ref{fig_finalDesigns1} and \ref{fig_finalDesigns2}, i.e. Figures \ref{fig_finalDesigns1}(a),(c) and Figures \ref{fig_finalDesigns2}(a),(c),(e), show the obtained element-wise material distribution from five different perspectives. Given this material distribution, we solved problem \eqref{eq_stateNumerics} by means of piecewise linear, globally continuous finite elements on the given grid to obtain the desired state $u^* = u_d$. For all numerical computations, we used the finite element software package \texttt{NGSolve} \cite{Schoeberl2014}.

As an initial design for the optimisation, we choose the empty set, $\Omega = \emptyset$ corresponding to a design where the sphere is only covered by water regions. This is realized by choosing $\psi_0 =1 / \|1 \|_{L_2(M)}$ as the initial level set function. This level set function is updated according to \eqref{eq_levelsetupdate} by means of the generalized topological derivative. Figure \ref{fig_TDs}(b)--(f), shows the topological derivative according to formula \eqref{eq_TDfinal} on $M$ for this initial configuration from five different angles.

In our numerical experiments, we used a rather conservative choice of the line search parameter $\kappa$: We initialized it to $\kappa = \kappa_{max} := 0.05$. When no decrease was achieved with this value, we halved $\kappa$ until the objective function decreased. At the end of each iteration, we increased $\kappa$ by a factor of $1.1$ and projected the resulting value to $[0, \kappa_{max}]$.

 After 57 iterations of the optimization algorithm, the objective function was reduced from approximately $2.5 \cdot 10^6$ to approximately $2 \cdot 10^3$. We remark that, due to the fine-scale topographical data used in this example, the limited computational resources and the fact that the optimal design is given as element data and therefore not smooth, the exact optimiser could not be reached.
 In order to obtain better accuracy at the material interfaces, an approach incorporating shape sensitivity information could be used. However, this is beyond the scope of this paper.
 Nevertheless, Figures \ref{fig_finalDesigns1} and \ref{fig_finalDesigns2} show that the reconstruction was successful and all of the land masses could be recovered to a rather good precision, which illustrates the usefulness of topological derivatives in topology optimisation problems posed on manifolds.

\begin{figure}
    \begin{tabular}{cc}
        \includegraphics[width=.5\textwidth, trim = 450 0 330 0, clip]{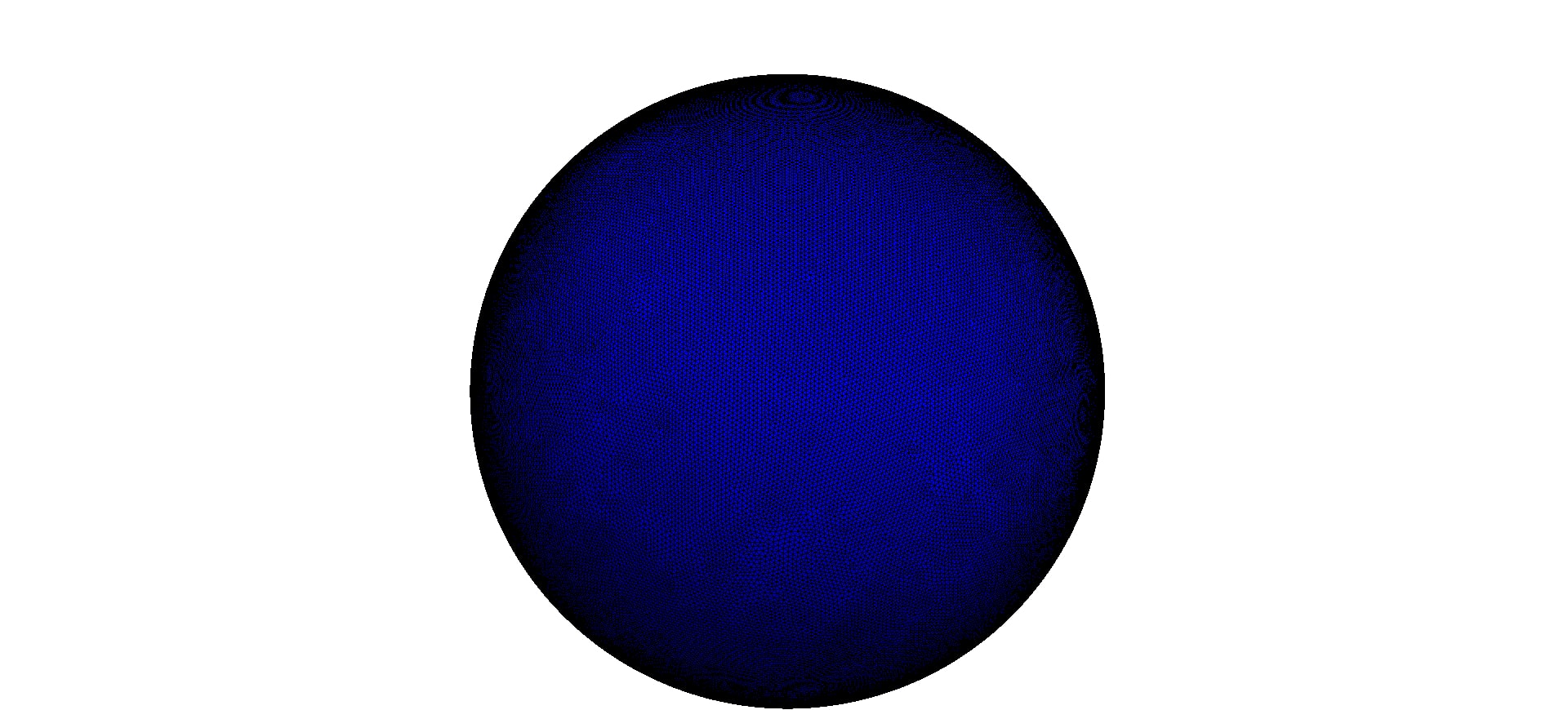} &         \includegraphics[width=.5\textwidth, trim = 450 0 330 0, clip]{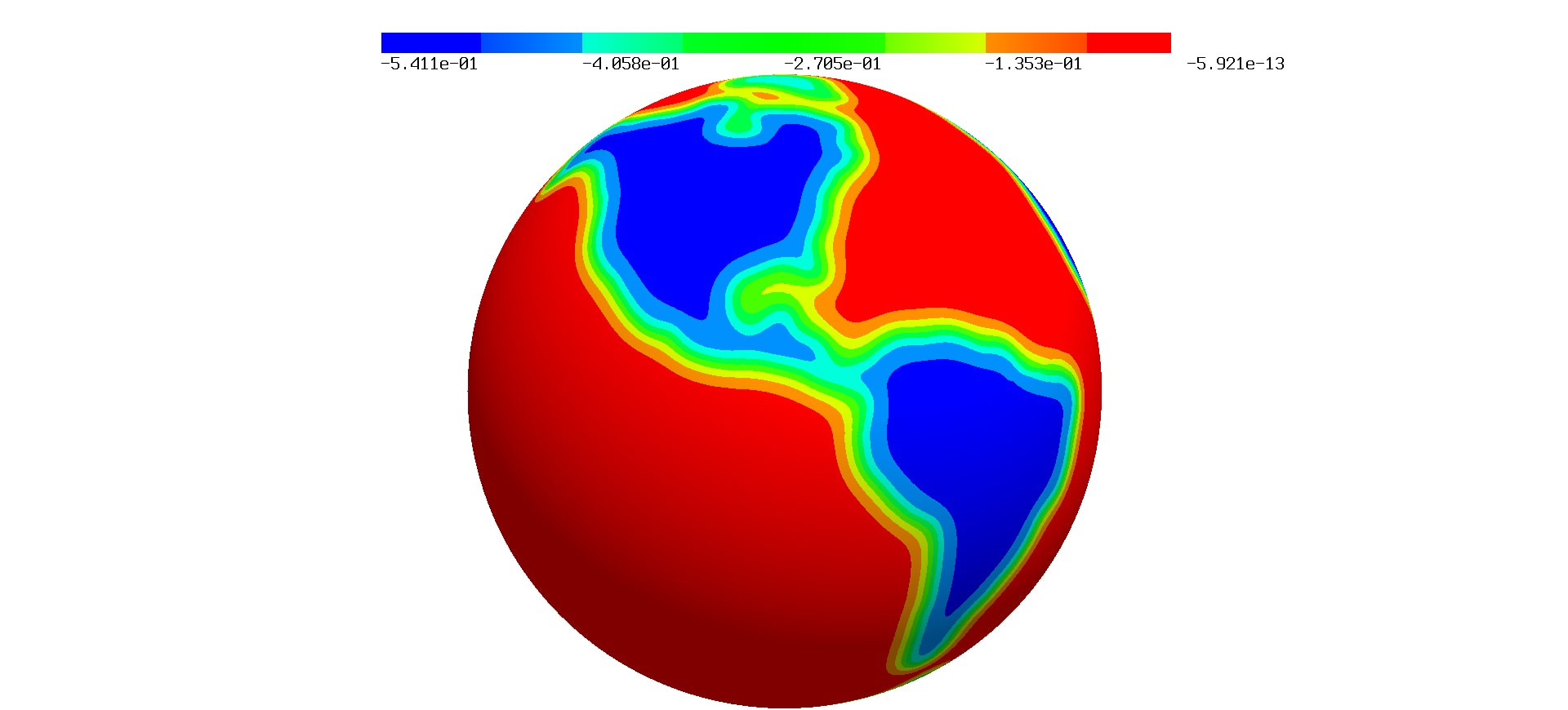} \\ (a) & (b) \\
        \includegraphics[width=.5\textwidth, trim = 450 0 330 0, clip]{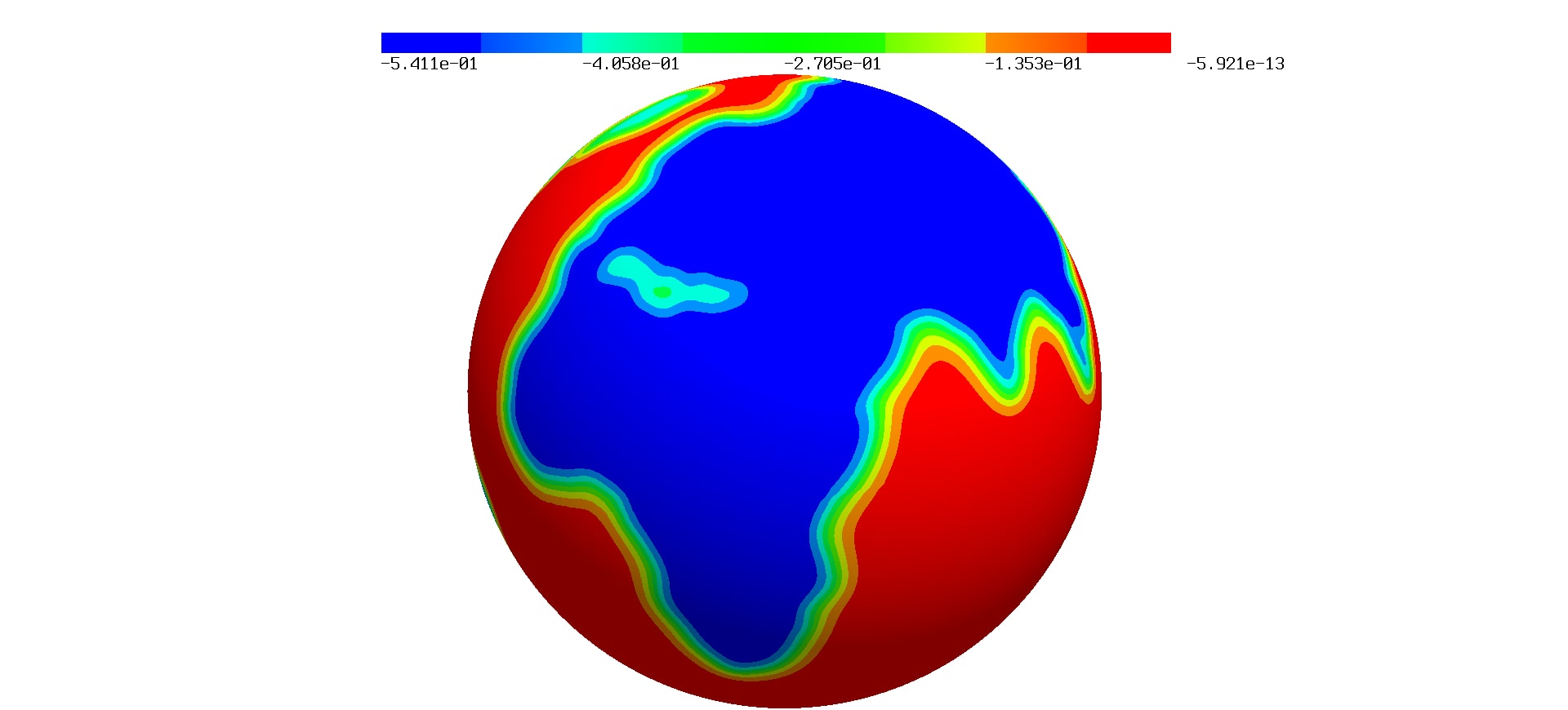} &         \includegraphics[width=.5\textwidth, trim = 450 0 330 0, clip]{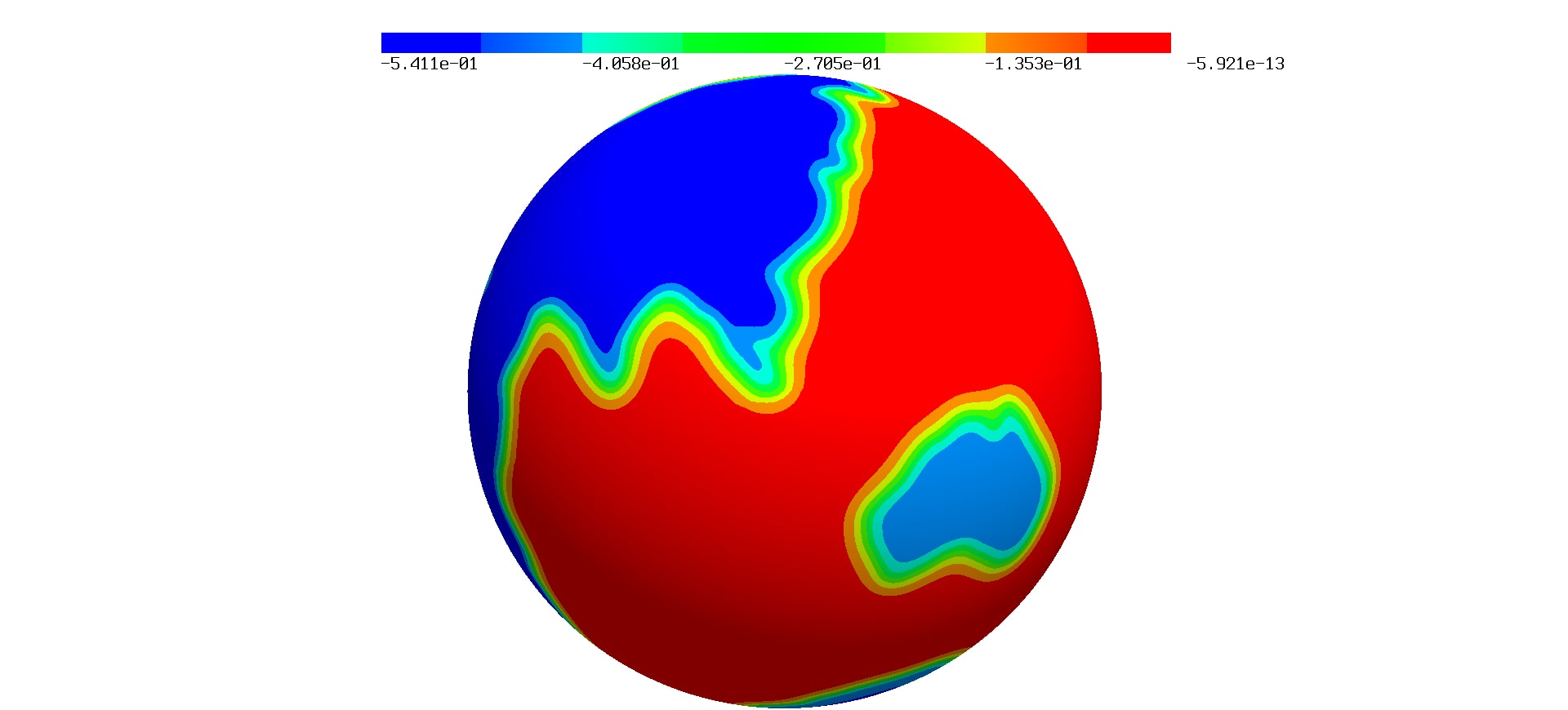} \\(c) & (d) \\
        \includegraphics[width=.5\textwidth, trim = 450 0 330 0, clip]{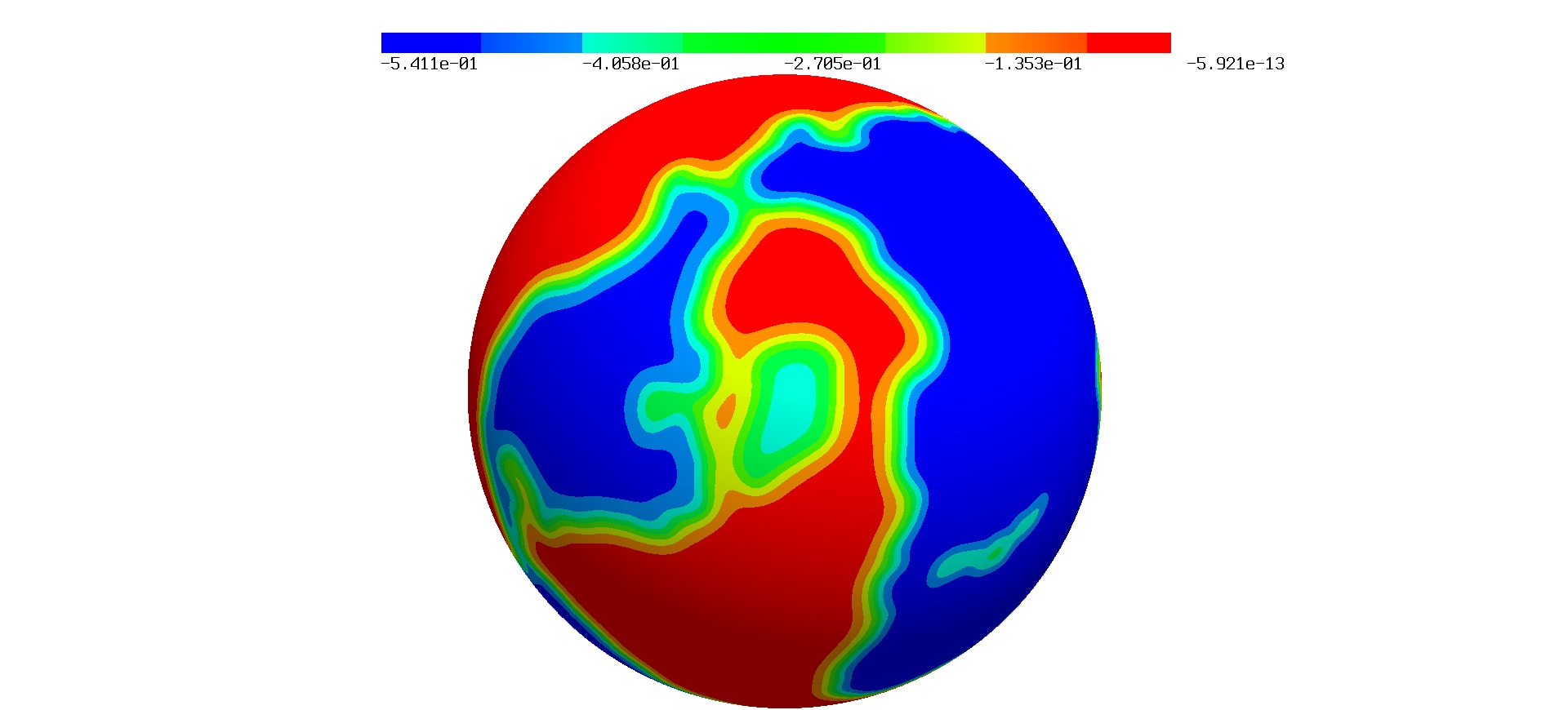} &         \includegraphics[width=.5\textwidth, trim = 450 0 330 0, clip]{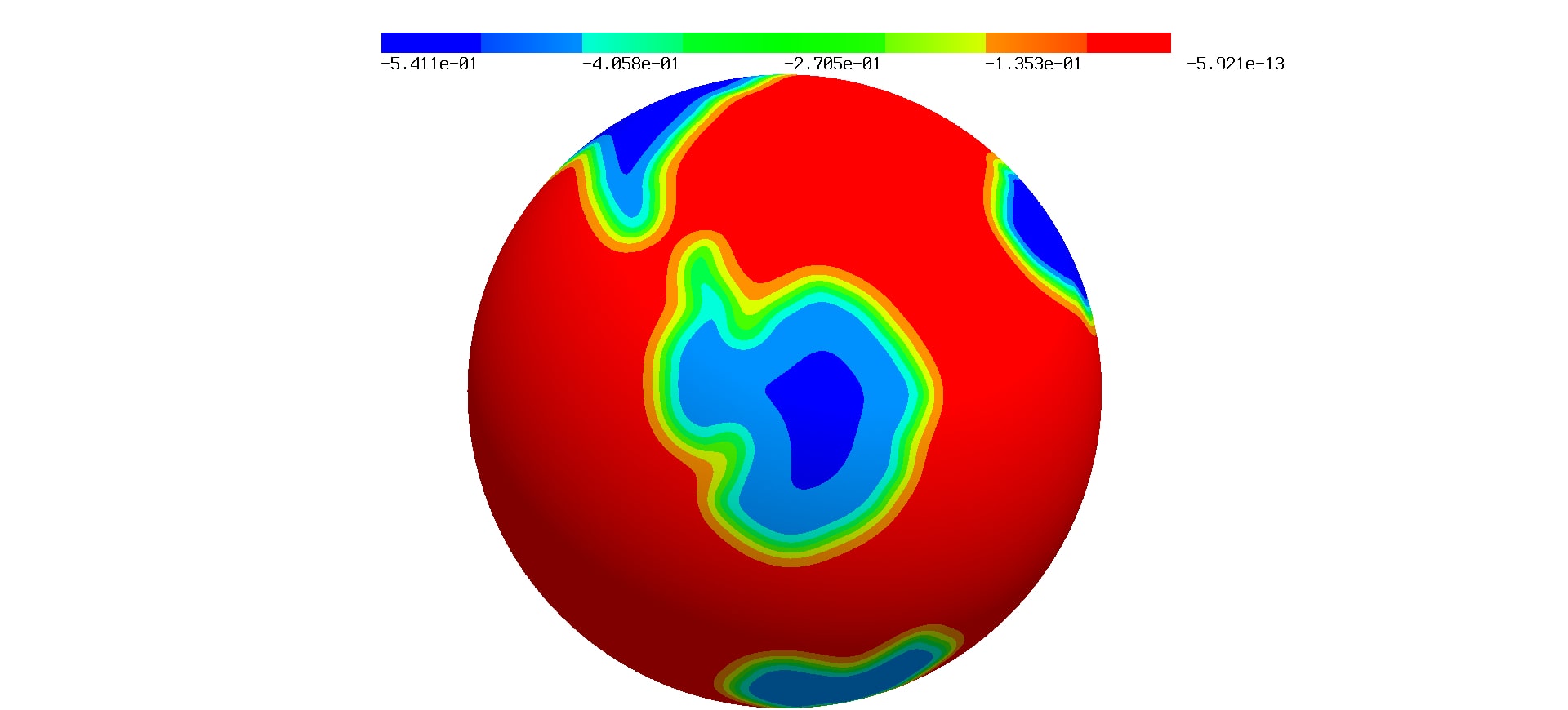}  \\(e) & (f) 
    \end{tabular}
    \caption{(a) Initialialization of level set function as constant. (b)--(f) Different views of topological derivative for initial configuration.}
    \label{fig_TDs}
\end{figure}

\begin{figure}
    \begin{center}
    \begin{tabular}{c|c}
        \includegraphics[width=.5\textwidth, trim = 450 0 450 0, clip]{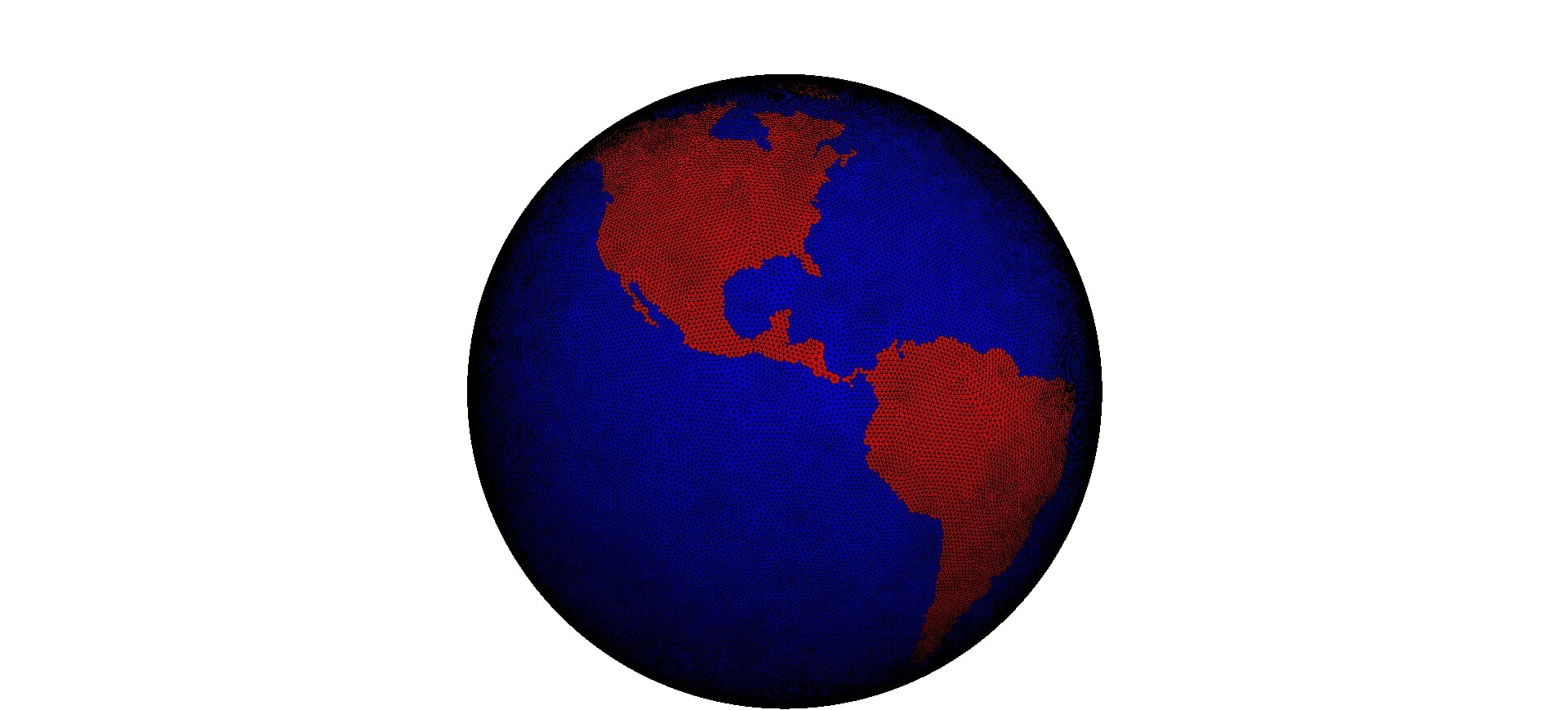} &         \includegraphics[width=.5\textwidth, trim = 450 0 450 0, clip]{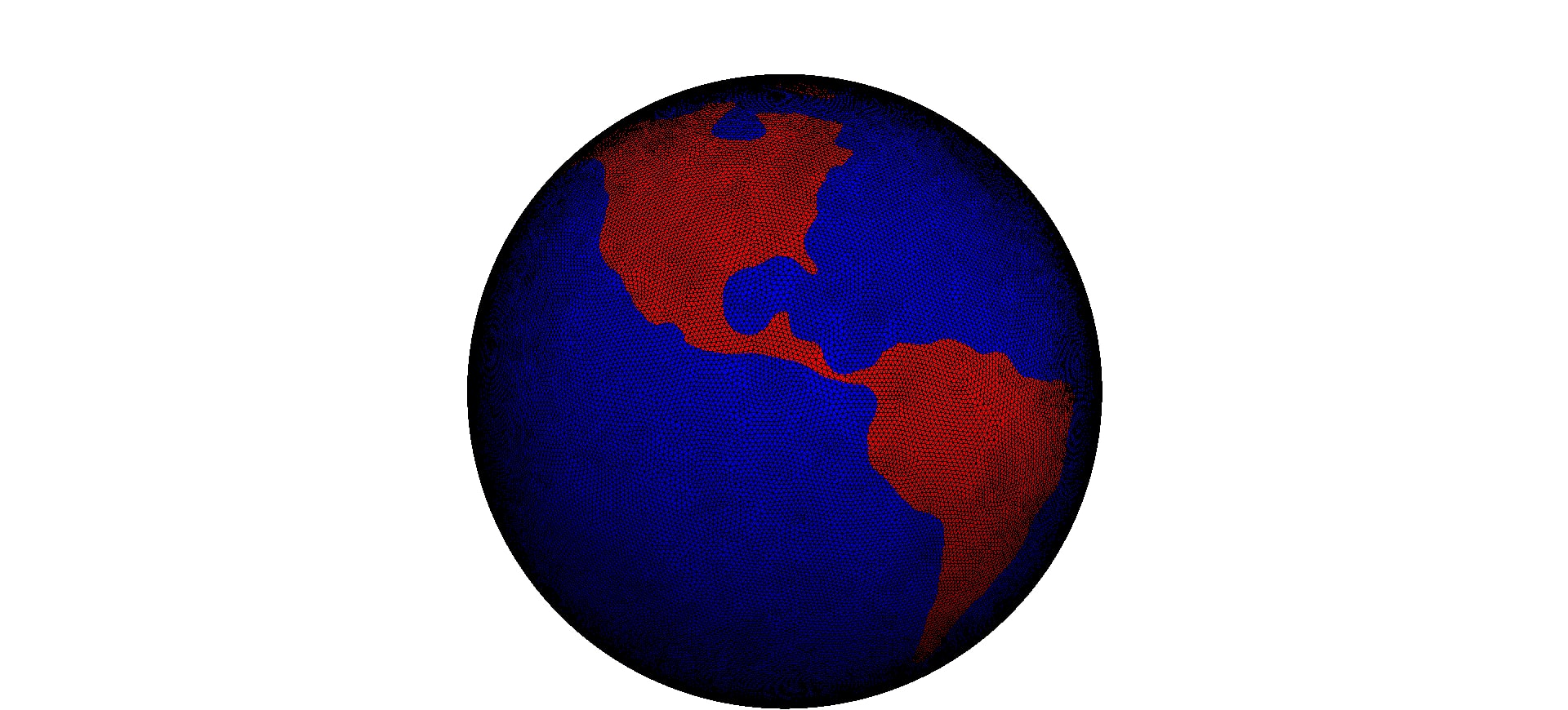} \\
        (a)&(b)\\
        \includegraphics[width=.5\textwidth, trim = 450 0 450 0, clip]{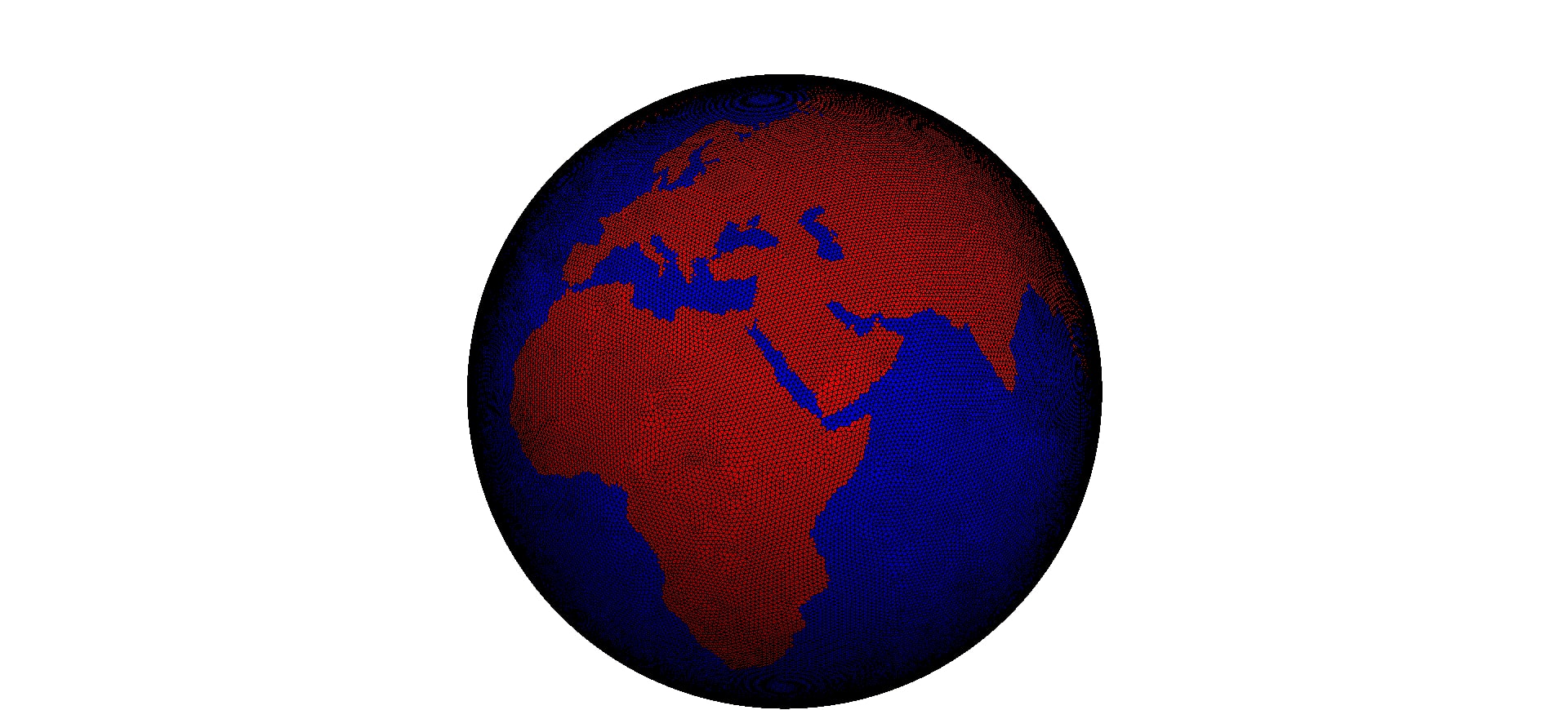} &         \includegraphics[width=.5\textwidth, trim = 450 0 450 0, clip]{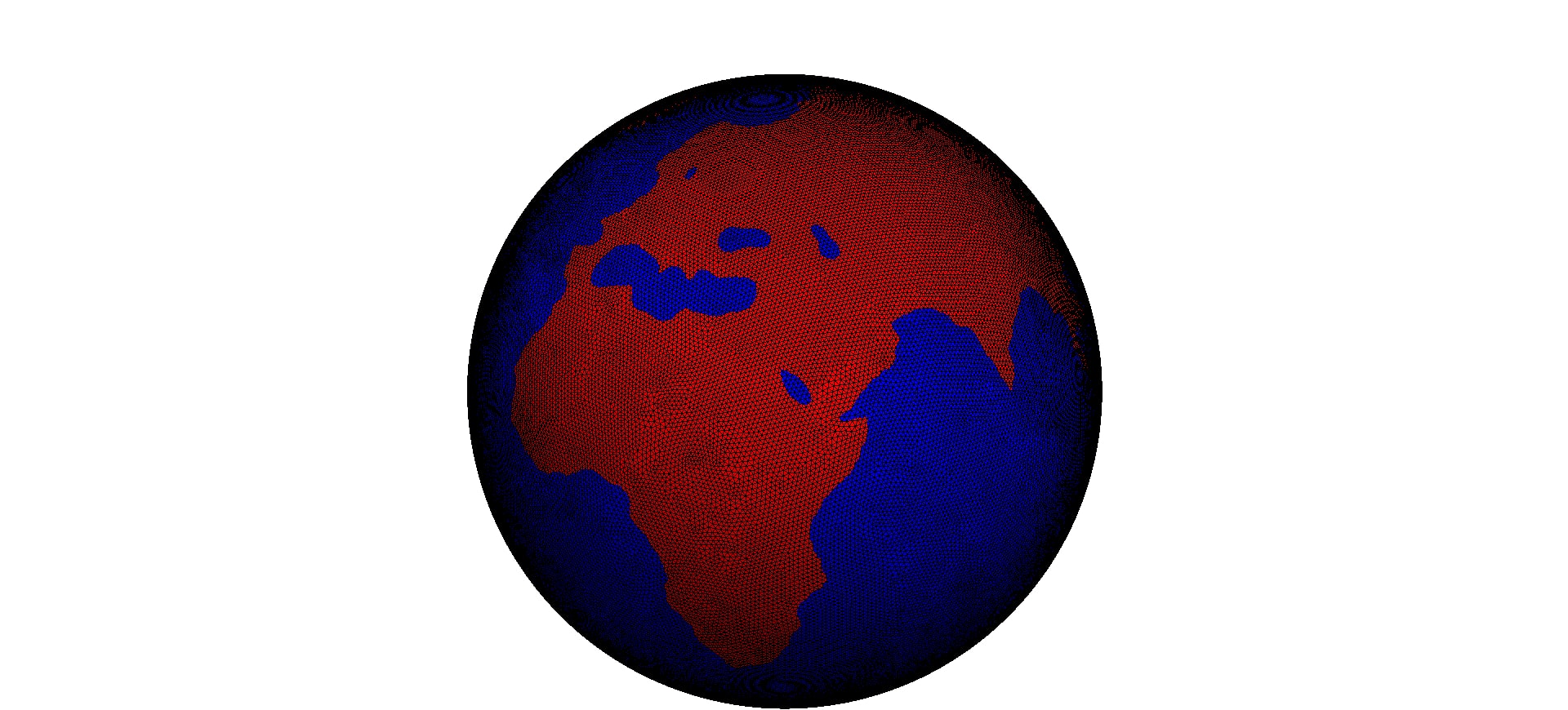} \\
        (c)&(d)
    \end{tabular}
    \end{center}
    \caption{Different views of desired and final geometry. Left column: desired material distribution $\Omega^*$. Right column: Material distribution obtained after 57 iterations of level set algorithm to \eqref{eq_problemNumerics} where $u_d$ is the numerical solution to \eqref{eq_stateNumerics} with $\Omega = \Omega^*$. }
    \label{fig_finalDesigns1}
\end{figure}

\begin{figure}
    \begin{center}
    \begin{tabular}{c|c}
        \includegraphics[width=.4\textwidth, trim = 450 0 450 0, clip]{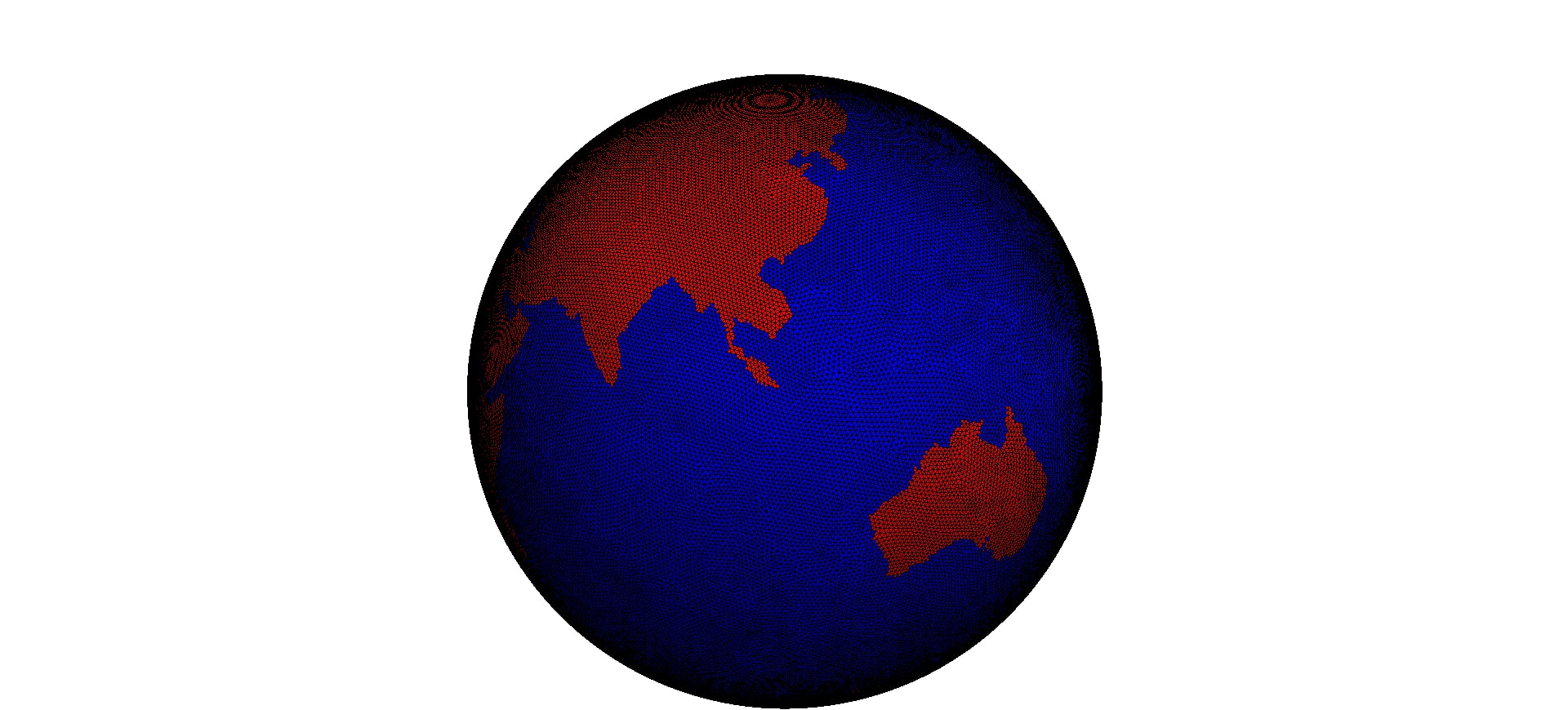} &         \includegraphics[width=.4\textwidth, trim = 450 0 450 0, clip]{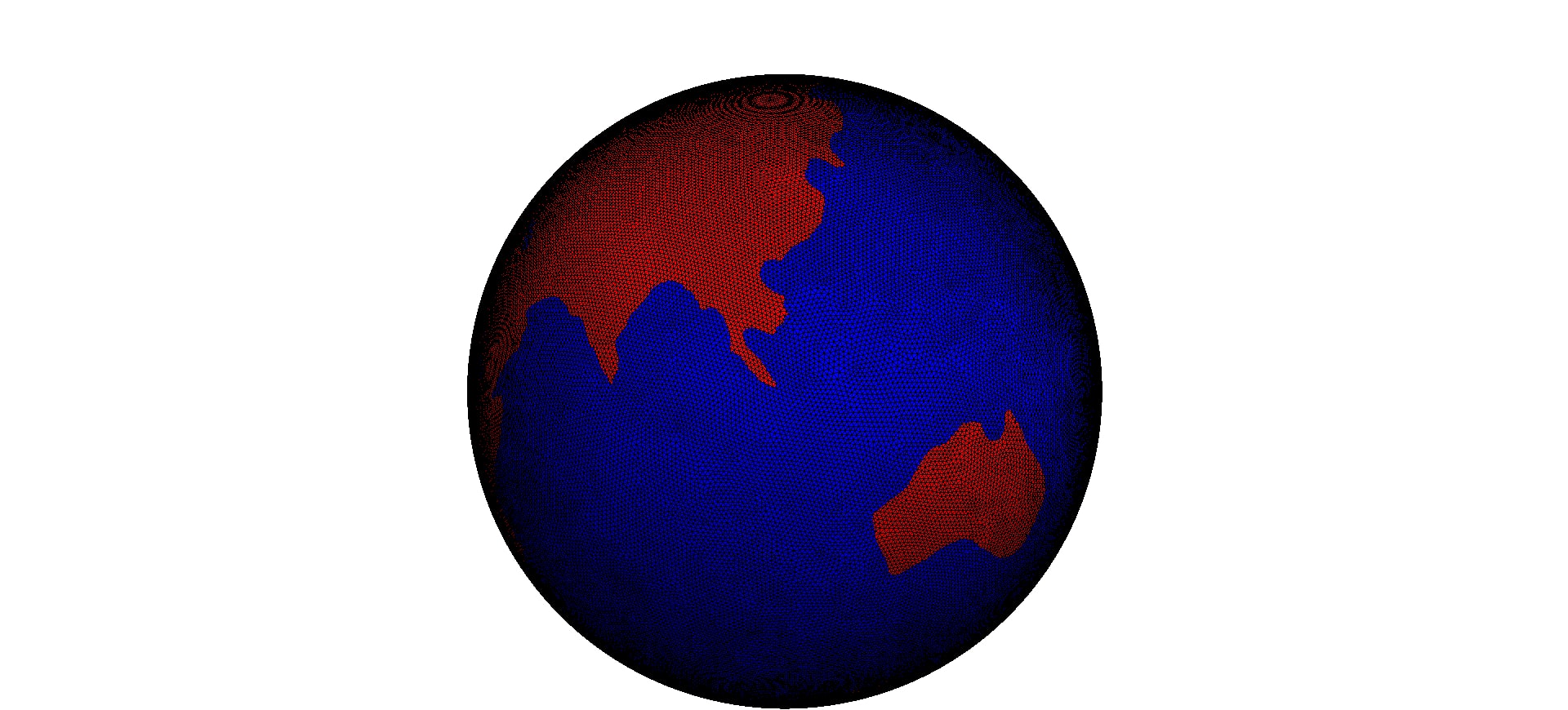} \\
        (a)&(b)\\
        \includegraphics[width=.4\textwidth, trim = 450 0 450 0, clip]{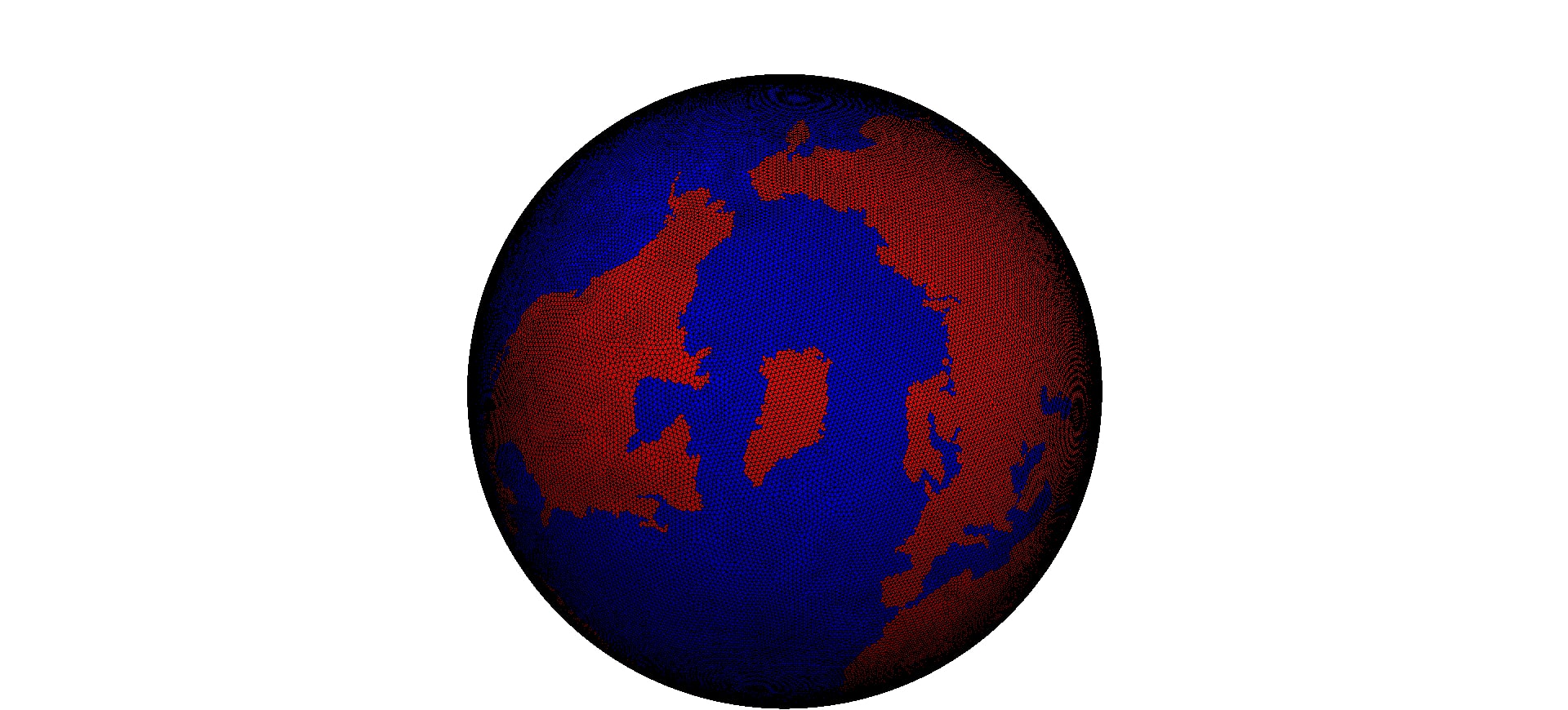} &         \includegraphics[width=.4\textwidth, trim = 450 0 450 0, clip]{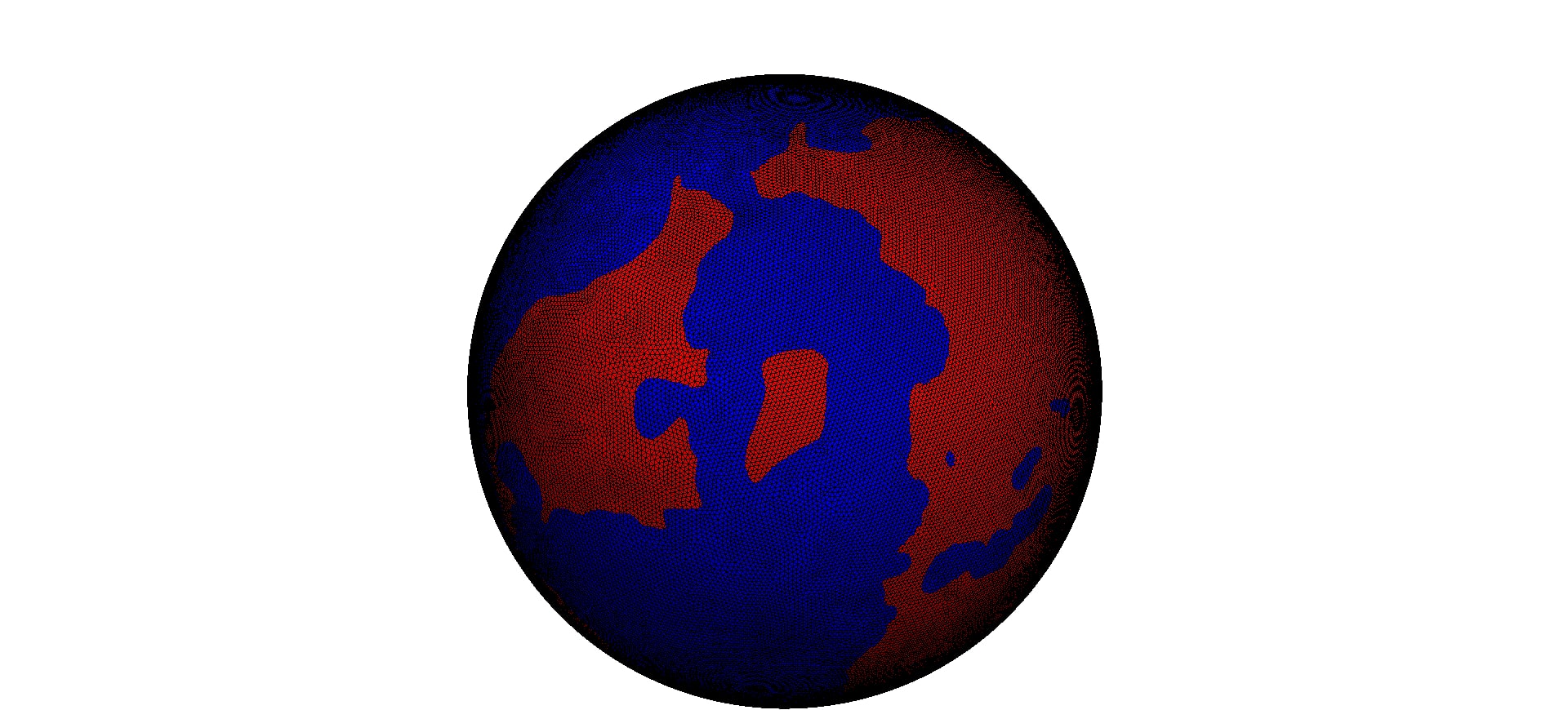} \\
        (c)&(d)\\
        \includegraphics[width=.4\textwidth, trim = 450 0 450 0, clip]{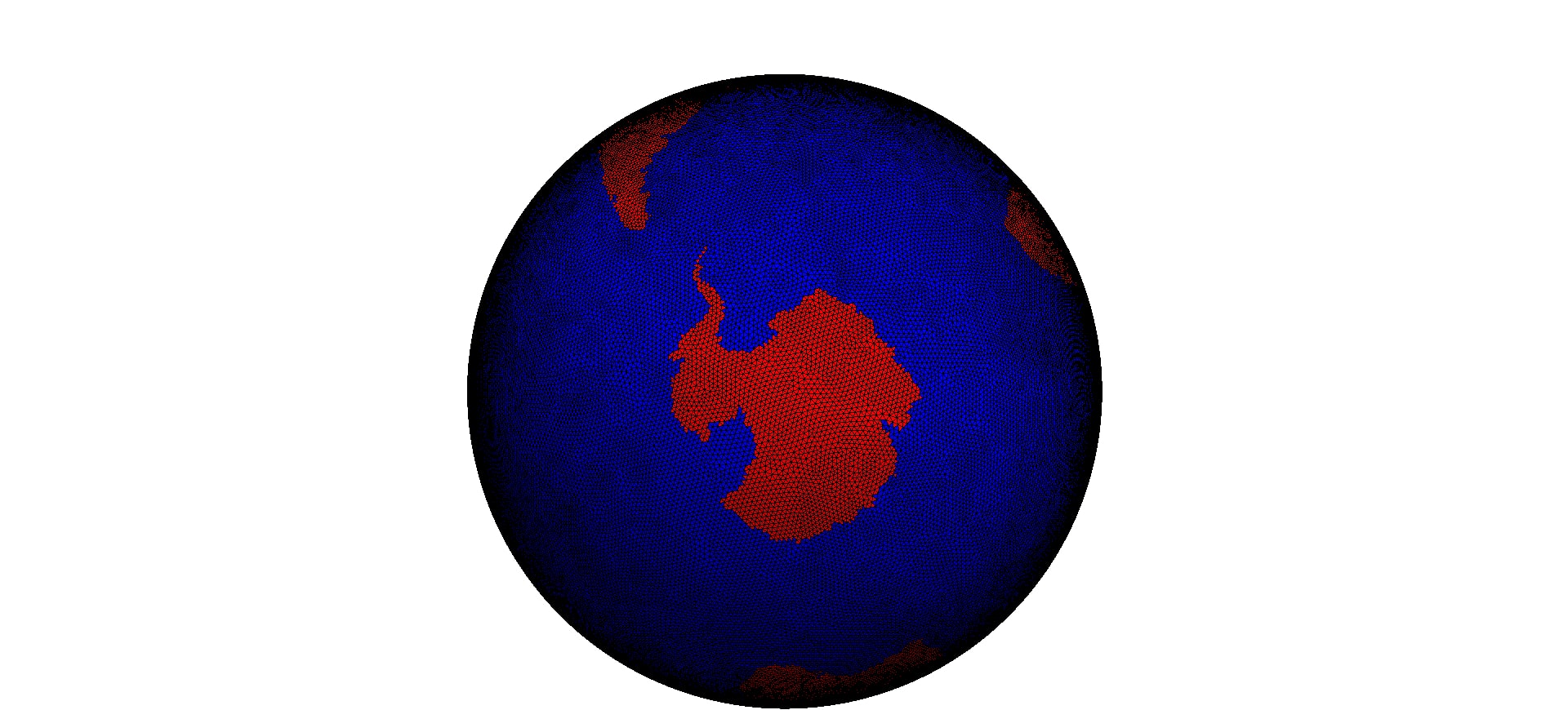} &         \includegraphics[width=.4\textwidth, trim = 450 0 450 0, clip]{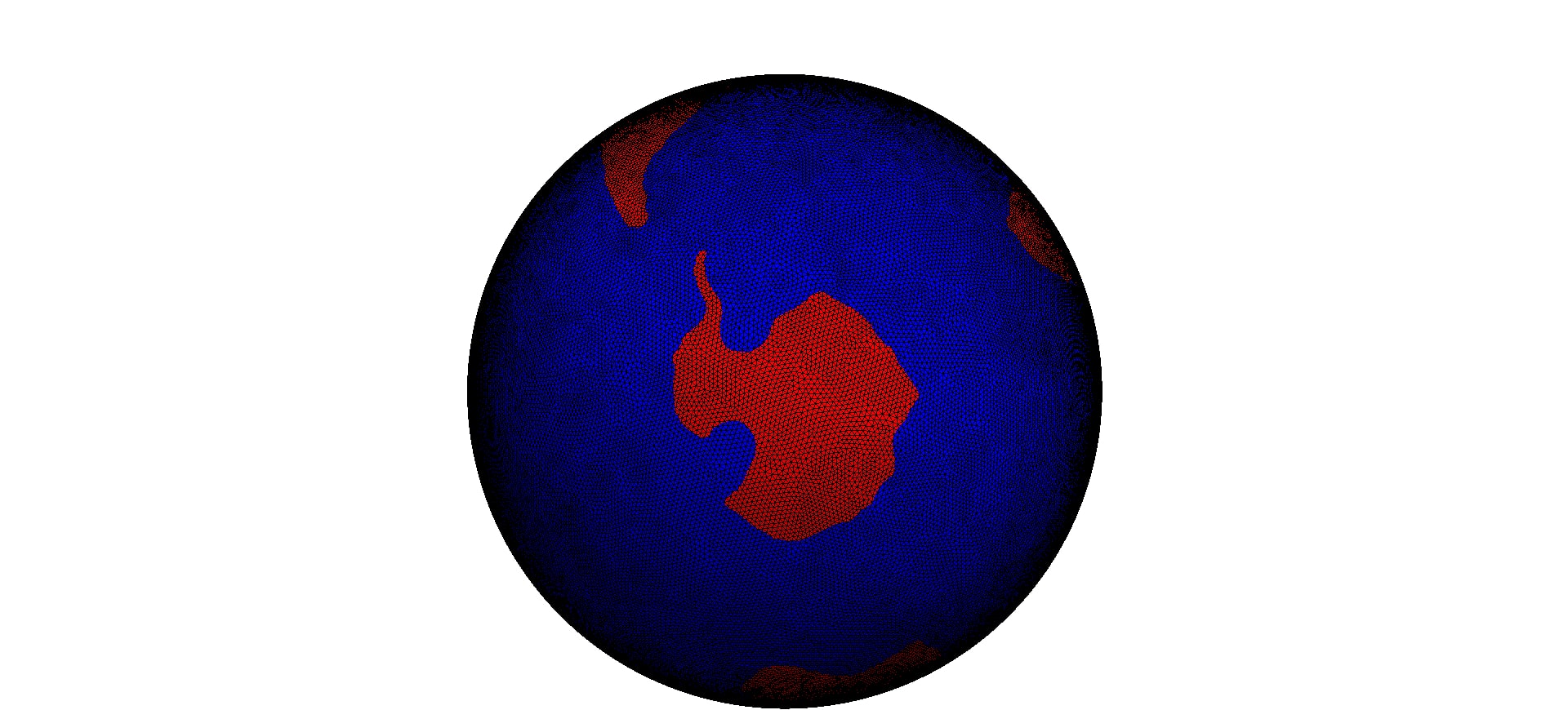} \\
        (e)&(f)
    \end{tabular}
    \end{center}
     \caption{Different views of desired and final geometry. Left column: desired material distribution $\Omega^*$. Right column: Material distribution obtained after 57 iterations of level set algorithm to \eqref{eq_problemNumerics} where $u_d$ is the numerical solution to \eqref{eq_stateNumerics} with $\Omega = \Omega^*$. }
    \label{fig_finalDesigns2}
\end{figure}

\section*{Conclusion}
In this paper we derived for the first time topological sensitivities for PDEs defined on surfaces. We showed how the sensitivities can be used in a level set algorithm on the surface and showed its performance in a numerical experiment. Our techniques open now the possibilities to derive sensitivities for other types of more general surface PDE such as the Laplace-Betrami equation involving differential forms, which will be part of future research. Another important issue which we will address in the future is the higher asymptotic expansion of the state equation of the surface and thus extending the results 
of Lemma~\ref{thm_Keps_K}. Lastly, another important question open for further research is the treatment of nonlinear equations on surfaces. Due to the non-linearity of manifolds this poses new and interesting challenges that could be addressed in future work.

\bibliography{topDerSurface_GanglSturm}
\bibliographystyle{plain}
\end{document}

%% file: slidesymbols.tex
\newcommand{\VR}{{\mathbf{R}}}

\newcommand{\Dsf}{\mathsf{D}}

\providecommand{\Cj}{{\cal J}}

\providecommand{\Cp}{{\cal P}}